\documentclass[11pt]{amsart}
\usepackage[utf8]{inputenc}
\usepackage{amsthm}
\usepackage{amsmath}
\usepackage{amssymb}
\usepackage{hyperref}
\usepackage{amsrefs}
\newtheorem{theorem}{Theorem}[section]

\newtheorem{lemma}[theorem]{Lemma}

\newtheorem{definition}[theorem]{Definition}

\newtheorem{remark}[theorem]{Remark}

\newtheorem{proposition}[theorem]{Proposition}

\numberwithin{equation}{section}

\newcommand{\bb}[1]{\mathbb{#1}}
\newcommand{\cl}[1]{\mathcal{#1}}
\begin{document}
\title[]{Perturbations of Toeplitz operators on vector-valued Hardy spaces}

\author{Arshad Khan}
\address{Department Of Mathematics\\
         School of Natural Sciences\\
        Shiv Nadar Institution Of Eminence\\
        Gautam Buddha Nagar - 201314\\
         Uttar Pradesh, India}
\email{ak954@snu.edu.in}

\author{Sneh Lata}
\address{Department Of Mathematics\\
         School of Natural Sciences\\
         Shiv Nadar Institution Of Eminence\\
         Gautam Buddha Nagar - 201314\\
         Uttar Pradesh, India}
\email{sneh.lata@snu.edu.in}

\author{ Dinesh Singh}
\address{Centre For Digital Sciences\\
      O. P. Jindal Global University\\
   Sonipat\\
        Haryana 131001, India}
\email{dineshsingh1@gmail.com}



\begin{abstract}
In this article, we completely classify invariant subspaces of finite-rank perturbations of a class of Toeplitz operators on vector-valued Hardy spaces. As a consequence, 
in the vector-valued setting, we characterize invariant and almost invariant subspaces of a class of Toeplitz operators, as well as nearly invariant subspaces associated with certain Blaschke-based operators. We further treat the finite defect case for these nearly invariant subspaces.
\end{abstract}

\maketitle

\section{Introduction} Perturbations of Toeplitz operators on Hardy spaces have engendered a great deal of interest because–as is well known–they help us understand how small changes in these operators affect their properties, like spectrum or invertibility, which are crucial for solving equations in signal processing and control theory. They’re also key in studying stability in function spaces, where even tiny tweaks can impact convergence or approximation in applications like harmonic analysis or quantum mechanics. The basic ideas in this context are available in the volumes by B\"{o}ttcher and Silbermann \cite{BS} and Rosenblum and 
Rovnyak \cite{RR1}. We also refer to the paper by Axler, Chang, and Sarason \cite{AS}.  

In this article, we explore and establish interesting properties and connections between perturbation theory, invariant subspaces, nearly invariant subspaces and almost 
invariant subspaces in the context of certain Toeplitz operators on vector-valued Hardy spaces $H^2(\bb D, \bb C^m)$. See Section \ref{state} for the statements of our main results. It is worth mentioning that Hitt \cite{Hit} introduced nearly invariant subspaces but called them \textit{weakly invariant}. Subsequently, Sarason \cite{Sar} rechristened them with the now universally accepted label \textit{nearly invariant} subspaces and he introduced new ideas into their study by establishing a new proof via de Branges-Rovnyak spaces. A subspace $\cl M$ of the Hardy space $H^2(\bb D)$ is called nearly invariant under the backward shift operator $T_z^*$ on $H^2(\bb D)$ if $T_z^*(f)$ belongs to 
$\cl M$ whenever $f\in \cl M$ and vanishes at zero. Since then, numerous generalizations have been proposed in this direction. For further insights into 
nearly invariant subspaces, one can refer to \cite{CCP}, \cite{CGP}, \cite{CD}, \cite{CDP1}, \cite{DS}, \cite{Era}, \cite{KLD}, \cite{LP}, \cite{LP1}, \cite{LP2}, \cite{Sar} and the references therein. We also refer to two recent related papers \cite{CDP2} and \cite{LP3}.  

A transformative feature of nearly invariant subspaces of the backward shift operator on $H^2(\bb D)$ observed and used by Sarason is their connection with a particular rank-one perturbation of the backward shift operator. Inspired by this masterful insight of Sarson, Das and Sarkar \cite{DS} have recently investigated the invariant subspaces  
of certain Toeplitz operators, a generalization of nearly invariant subspaces, and almost invariant subspaces on $H^2(\bb D)$ via a classification of invariant subspaces of some finite-rank perturbations of certain Toeplitz operators that they derive in the same paper.  

Our work is inspired by the perturbative perspective developed in \cite{DS}, and we elaborate on this connection in the subsequent sections. In the present paper, we obtain a complete classification of the invariant subspaces of finite-rank perturbations of a class of Toeplitz operators on vector-valued Hardy spaces $H^2(\bb D, \bb C^m)$. In doing so, we uncover structural connections between perturbation theory, invariant subspaces, almost invariant subspaces, and nearly invariant subspaces in the vector-valued setting. While our approach is motivated by ideas introduced in \cite{DS} for the scalar case, the passage to the vector-valued framework involves additional technical considerations and a more intricate operator-theoretic structure. 

The organization of the rest of the paper is as follows. In Section \ref{Sec2}, we introduce the notations, and preliminaries used throughout the article. Section \ref{state} contains the statements of our main results.  

In Section \ref{Sec3}, we prove Theorems \ref{Th3.2} and \ref{Th3.6}, which provide a classification of invariant subspaces for finite-rank perturbations of a class of Toeplitz operators on the vector-valued Hardy space $H^2(\bb{D}, \bb{C}^m)$. We also establish three consequences of these results:  
\begin{enumerate}
\item[$(i)$] a characterization of invariant subspaces of finite-rank perturbations of a class of analytic Toeplitz operators on $H^2(\mathbb{D}, \mathbb{C}^m)$ (Theorem \ref{Th3.7}); 
\item[$(ii)$] a characterization of almost invariant subspaces of a class of Toeplitz operators on $H^2(\mathbb{D}, \mathbb{C}^m)$ (Theorem \ref{Th3.10});  
\item[$(iii)$] a characterization of nearly invariant subspaces associated with certain Blaschke-based operators on $H^2(\mathbb{D}, \mathbb{C}^m)$ (Theorem \ref{Th3.13}). 
\end{enumerate}

Section \ref{Sec4} is devoted to extending Theorem \ref{Th3.13} to the finite defect case (Theorems \ref{Th4.2}, \ref{Th4.4}, and \ref{Th4.5}). This finite defect analysis does not appear in \cite{DS}.  

Sections \ref{Sec3} and \ref{Sec4} also contain several significant remarks and crucial observations that highlight essential aspects of our findings.

\section{Preliminaries}\label{Sec2}
Let $\mathbb{D}$ be the open unit disc in the complex plane $\mathbb{C}$. Then the $\mathbb{C}^n$-valued Hardy space, denoted by 
$H^2(\bb{D}, \bb{C}^n)$, is the Hilbert space of $\bb {C}^n$-valued analytic functions defined on 
$\bb D$ with square summable Fourier coefficients in their power series representations. That is;
\begin{eqnarray*}
 H^2(\mathbb{D},\bb C^n):=\Bigg\{F=
 \sum_{i=0}^{\infty}A_iz^i : \ A_i \in \bb C^n, \ \sum_{i=0}^{\infty}\|A_i\|_{\bb C^n}^2 < \infty \Bigg\}.   
\end{eqnarray*}

For any $F(z)=\sum_{i=0}^{\infty}A_i z^i $ and $H(z)=
 \sum_{i=0}^{\infty}B_i z^i$ in $H^2(\mathbb{D},\bb C^n)$, the inner product of $F$ and $H$ is given 
 by
\begin{eqnarray*}
\langle{F,H}\rangle=\sum_{i=0}^{\infty}\langle{A_i, B_i}\rangle_{\bb C^n}.
\end{eqnarray*}

We shall use $\{E_i: 1\le i \le n\}$ to denote the standard orthonormal basis of $\bb C^n$. Note that each $E_i$ can be viewed as a function in $H^2(\bb D, \bb C^n)$ by identifying it with the constant function that is identically $E_i$. Furthermore, $\{E_jz^m:1\le j\le n, \ m\ge 0\}$ forms an orthonormal basis of the Hardy space $H^2(\mathbb{D}, \bb C^n).$ For $n=1$, we denote $H^2(\mathbb{D},\bb C^n)$ by $H^2(\mathbb{D})$.  

Throughout this paper, we shall frequently identify $H^2(\mathbb{D}, \bb C^n)$ with the Hilbert space tensor product $H^2(\mathbb{D})\otimes \bb C^n$ of $H^2(\mathbb{D})$ and $\bb C^n$. Indeed, the natural map $E_jz^m\mapsto E_j\otimes z^m$ extends to a unitary from $H^2(\mathbb{D}, \bb C^n)$ onto $H^2(\mathbb{D})\otimes \bb C^n$. One more identification of the vector-valued Hardy space is crucial for our present work. Suppose $p$ and $m$ are two natural numbers, then the Hardy spaces $H^2(\mathbb{D}, \mathbb{C}^{p+m})$ and $H^2(\mathbb{D}, \mathbb{C}^p)\oplus H^2(\mathbb{D}, \mathbb{C}^m)$ are unitarily equivalent, where the unitary is given by the map 
$\begin{pmatrix}
A \\ 
C
\end{pmatrix}z^n
\mapsto 
\begin{pmatrix}
 A z^n\\ 
C z^n
\end{pmatrix},$ where $A\in \mathbb{C}^p$ and $C\in \mathbb{C}^m$. Although under the latter identification, we write elements of $H^2(\bb D, \bb C^{p+m})$ as columns 
$\begin{pmatrix}
F\\H
\end{pmatrix}
$ with $F\in H^2(\bb D, \bb C^p)$ and $H\in H^2(\bb D, \bb C^m)$, frequently we also write them, for notational convenience, simply as tuples $(F, H)$. We shall interchangeably use the identification of the Hardy spaces and the notations without mention.

The latter identification allows one to view operators between these vector-valued Hardy spaces as row matrices. Suppose $\cl G$ is a bounded operator from $H^2(\mathbb{D}, \mathbb{C}^p)$ to $H^2(\mathbb{D}, \mathbb{C}^m)$ and $\cl J$ is a bounded operator from $H^2(\mathbb{D},\mathbb{C}^r)$ to 
$H^2(\mathbb{D}, \mathbb{C}^m)$, then 
we shall use the notation $[\cl G, \cl J]$ to denote the bounded operator from $H^2(\mathbb{D}, \mathbb{C}^{p+r})$ to $H^2(\mathbb{D}, \mathbb{C}^m)$ defined by  
$$
\begin{pmatrix}
F \\ 
H
\end{pmatrix}\mapsto \cl GF+\cl J H.
$$

Further, suppose $\{G_i:1\le i\le m\}$ is a set in $H^2(\mathbb{D}, \mathbb{C}^n)$ such that for each 
$f_1, \dots, f_m\in H^2(\mathbb{D})$, the function $\sum_{i=1}^m f_iG_i$ belongs to $H^2(\mathbb{D}, \mathbb{C}^n)$. Then we denote the map 
$$
\begin{pmatrix}
f_1 \\
\vdots\\
f_m
\end{pmatrix}
\mapsto 
\sum_{i=1}^mf_iG_i
$$ 
by a row matrix $[G_1, \dots, G_m]$, and call it a matrix with $m$ columns $G_1, \dots, G_m$.

We now turn to some specific bounded operators that we deal with in this work. The most essential non-trivial bounded operator that we study on $H^2(\mathbb{D})$ is the $\textit{Shift operator}$.
\begin{definition}
The Shift operator on $H^2(\mathbb{D})$, denoted by $T_z$, is defined as
\begin{eqnarray*}
(T_zf)(w)=wf(w), \quad f\in H^2(\mathbb{D}).
\end{eqnarray*}
The adjoint of $T_z$, denoted by $T^*_z$, is called the backward shift operator. It is given by
\begin{eqnarray*}
    (T^*_zf)(w)=\frac{f(w)-f(0)}{w}, \quad f\in H^2(\mathbb{D}).
\end{eqnarray*}
\end{definition}\label{Def2.1}

Let $B(\bb C^n,\bb C^m)$ denote the Banach algebra of bounded operators from $\bb C^n$ to $\bb C^m$. 
Then we denote the Banach algebra of $B(\bb C^n,\bb C^m)$-valued bounded analytic functions on 
$\mathbb{D}$ by $H^{\infty}(\mathbb{D},B(\bb C^n, \bb C^m))$, and 
call its elements \textit{multipliers}. Each multiplier $\Phi$ in 
$H^{\infty}(\mathbb{D},B(\bb C^n, \bb C^m))$ gives rise to a bounded operator 
$T_\Phi :H^2(\mathbb{D}, \bb C^n) \longrightarrow H^2(\mathbb{D}, \bb C^m)$ given by $$(T_{\Phi}f)(z)=\Phi(z)f(z),$$ known as the \textit{Toeplitz operator} with symbol $\Phi$. Its adjoint is denoted by $T^*_\Phi$. If $\phi\in H^\infty(\mathbb{D})$, then the operator $T_\phi\otimes I_{\bb C^n}$ is the Toeplitz operator on 
$H^2(\mathbb{D}, \bb C^n)$ induced by the multiplier $\Phi(z)(z^n E_j)=\phi(z)z^nE_j.$
In this paper, we shall mainly deal with these Toeplitz operators. Note that for $m=n=1$ and $\phi(z)=z,$ the Toeplitz operator with symbol $\phi$ is simply the shift operator $T_z$ on $H^2(\mathbb D).$ Lastly, a function in $H^\infty(\mathbb{D})$ is said to be an \textit{inner function} if the corresponding Toeplitz operator $T_\phi$ on $H^2(\bb D)$ is an isometry.

A bounded operator $T$ on a Hilbert space $\mathcal H$ is called a $C._0$ \textit{contraction} if 
for each $h\in \mathcal H$, $\|T^{*n}h\| \longrightarrow 0$ as $n \longrightarrow \infty$. A 
$C._0$ isometry is called a \textit{pure isometry}. We end this section by quoting an interesting result of Benhida and Timotin on $C._0$ \textit{contraction} \cite{BT} that is indispensable for our work in the following sections.

\begin{lemma}\cite[Lemma 3.3]{BT}\label{Lem2.2}
Suppose $T$ is a bounded operator on a Hilbert space $\mathcal{H}$. Let $T$ be a $C._{0}$ contraction such that the rank of $I-T^*T$ is finite, and let $\mathcal{M}\subset \mathcal{H}$ be a subspace of finite codimension. Then $TP_\mathcal{M}$ is a $C._{0}$ contraction, where $P_\mathcal{M}$ is the orthogonal projection of $\mathcal{H}$ onto $\mathcal{M}$.   
\end{lemma}

\section{Statements of the main results}\label{state}
The following is our description of invariant subspaces for certain finite-rank perturbations of Toeplitz operators on $H^2(\bb D, \bb C^m).$

\begin{theorem}\label{Th3.2}
Suppose  $\phi \in H^2(\mathbb{D})$ is an inner function that vanishes at zero; $\{U_i\}_{i=1}^{k}$ and $\{V_i\}_{i=1}^{k}$ are two finite sets in $H^2(\mathbb{D},\mathbb{C}^m)$. 
Let $\cl {M}$ be a non-zero subspace of $H^2(\mathbb{D},\mathbb{C}^m)$ that is invariant under $T^*_\Phi - \sum_{i=1}^{k}V_i \otimes U_i$, where $T_\Phi:=T_\phi  \otimes I_{\mathbb{C}^m}$. Then there exists a non-negative integer $p$ and a $( T^*_\phi \otimes I_{\mathbb{C}^{p+m}} )$-invariant subspace $\mathcal{K}$ of $H^2(\mathbb{D},\mathbb{C}^{p+m})$ such that
$$
\mathcal{M}=[\mathcal{G}, I_m]\mathcal{K},
$$
where $\mathcal{G}$ is the matrix with $p$ columns that form an orthonormal basis of the subspace $span\{P_{\mathcal{M}}U_i: 1 \le i \le k\}$, and $I_m$ is the identity operator on $H^2(\mathbb{D}, \mathbb{C}^m).$ Furthermore, for each $F= \mathcal{G}R + H \in \mathcal{M}$,
$$
\|F\|^2=\|R\|^2 + \|H\|^2.
$$
\end{theorem}

\noindent Next, we have a converse of Theorem \ref{Th3.2}. 

\begin{theorem}\label{Th3.6}
Suppose a subspace $\mathcal{M}$ of $H^2(\mathbb{D}, \mathbb{C}^m)$ has the representation 
$$
\mathcal{M}=[\mathcal{G}, I_m]\mathcal{K}, \ \ {\rm where}
$$
\begin{enumerate}
\item[(i)] $\mathcal{G}=[G_1, \dots, G_p]$ is the matrix with $\{G_i\}_{i=1}^{p}$ an orthonormal set in $\mathcal{M}$; 
\item [(ii)] $\mathcal{K}\subset \{(\tilde{R}\circ \phi, H): \tilde{R}\in H^2(\mathbb{D}, \mathbb{C}^p), \ H\in H^2(\mathbb{D}, \mathbb{C}^m)\}$ is a $(T^*_\phi \otimes I_{\mathbb{C}^{p+m}})$-invariant subspace for an inner function $\phi$ that vanishes at zero. 
\end{enumerate}
If the map 
$\begin{pmatrix}
R\\H
\end{pmatrix}
\mapsto \mathcal{G}R+H$ is a unitary from $\mathcal{K}$ onto $\mathcal{M}$, then $\mathcal{M}$ is invariant under $T^*_\Phi-\sum_{i=1}^{p}T^*_\Phi G_i \otimes G_i$, where  $T_\Phi:=T_\phi \otimes I_{\mathbb{C}^m}$.
\end{theorem}

The structural insights provided by Theorems \ref{Th3.2} and \ref{Th3.6} extend well beyond their immediate context and serve as a foundational tool for analyzing a variety of operator-theoretic phenomena on vector-valued Hardy spaces. We apply these to deduce the following three results.

\vspace{.2 cm} 

\noindent {\bf Application I: Invariant subspaces of perturbed analytic Toeplitz operators.}  

\begin{theorem}\label{Th3.7} Let $\phi$ be an inner function in $H^2(\mathbb{D})$ that vanishes at zero, and $\mathcal{M}$ be a subspace of $H^2(\mathbb{D},\mathbb{C}^m)$. Suppose $\{U_i\}_{i=1}^{k}$ and $\{V_i\}_{i=1}^{k}$ are two finite sets in $H^2(\mathbb{D},\mathbb{C}^m)$ such that $\mathcal{M}$ is invariant under $T_\Phi-\sum_{i=1}^{k} V_i \otimes U_i$, where $T_\Phi:= T_\phi \otimes  I_{\mathbb{C}^m}.$ Then there exist a non-negative integer $p$ and a subspace $\mathcal{N}$ of $H^2(\mathbb{D},\mathbb{C}^{p+m})$ 
invariant under $T_\phi \otimes I_{\mathbb{C}^{p+m}}$ with 
$\mathcal{N}^\perp \subset\{(\tilde{R}\circ \phi, H): \tilde{R}\in H^2(\mathbb{D}, \mathbb{C}^p), H\in H^2(\mathbb{D}, \mathbb{C}^m)\}$, along with a 
unitary $\cl U: \mathcal{N}^\perp \longrightarrow \mathcal{M}^\perp$ given by 
 \begin{equation*}
      \cl U
      \begin{pmatrix}
      R\\
      H
      \end{pmatrix}=\mathcal{G}R+H,
 \end{equation*}
where $\mathcal{G}$ is the matrix with $p$ columns that form an orthonormal basis of the subspace $span\{P_{\mathcal{M}^{\perp}}V_i: 1\le i \le k\}$.  
 
Conversely, suppose $\{F_i\}_{i=1}^{p}$ is an orthonormal set in $\mathcal{M}^\perp$, and $\mathcal{N}$ is a subspace of $H^2(\mathbb{D}, \mathbb{C}^{p+m})$ invariant under $T_\phi \otimes I_{\mathbb{C}^{p+m}}$ with $\mathcal{N}^\perp \subset\{(\tilde{R}\circ \phi, H): \tilde{R}\in H^2(\mathbb{D}, \mathbb{C}^p), H\in H^2(\mathbb{D}, \mathbb{C}^m)\}$, along with a unitary 
$\cl U: \mathcal{\mathcal{N}^{\perp}}\longrightarrow \mathcal{M}^{\perp}$ defined by
\begin{equation*}
      \cl U
      \begin{pmatrix}
      R\\
      H
      \end{pmatrix}=\mathcal{G}R+H,
 \end{equation*}
where $\mathcal{G}=[F_1, \dots, F_p].$ Then
$\mathcal{M}$ is invariant under $T_\Phi-\sum_{i=1}^{p}F_i \otimes T^*_\Phi F_i$.  

\vspace{.1 cm} 

Moreover, if we assume $\{F_i\}_{i=1}^{p}\subset H^{\infty}(\mathbb{D},\mathbb{C}^m)$, then $\mathcal{M}$ can be expressed as 
$$
\mathcal{M}=\Big\{F \in H^2(\mathbb{D},\mathbb{C}^m): (T^*_{F_1}F,\dots,T^*_ {F_p}F,F) \in \mathcal{N}\Big\},    
$$
where $T_{F_i}^*$ is adjoint of the bounded operator $T_{F_i}: H^2(\mathbb{D}) \longrightarrow H^2(\mathbb{D},\mathbb{C}^m)$ that is given by $(F_if)(z)=f(z)F_i(z)$.
\end{theorem}

\noindent{\bf Application II: Almost invariant subspaces.} 

\begin{definition}\label{Def3.6} A subspace $\mathcal{M}$ of a Hilbert space $\mathcal{H}$ is said to be almost invariant under a bounded operator $T$ on $\mathcal{H}$ if there exists a finite-dimensional subspace $\mathcal{F}$ of $\mathcal{H}$ (orthogonal to $\cl M$) such that
\begin{eqnarray} \label{almc}
T{\mathcal{M}} \subseteq \mathcal{M} \oplus \mathcal{F}.
\end{eqnarray}    
\end{definition}

A finite-dimensional subspace of the smallest dimension that satisfies (\ref{almc}) is called the \textit{defect space} of $\mathcal{M}$, and its dimension is called the \textit{defect} of 
$\mathcal{M}$.

\begin{theorem}\label{Th3.10}
Suppose $\phi \in H^2(\mathbb{D})$ is an inner function that vanishes at zero and $T_\Phi:=T_\phi \otimes  I_{\mathbb{C}^m}$. Let $\mathcal{M}$ be an almost $T^*_\Phi$-invariant subspace of $H^2(\mathbb{D},\mathbb{C}^m)$ with defect $n$, and $\{F_i: i=1,2,...,n\}$ be an orthonormal basis of the defect space. Then there exists a non-negative integer $p$ and a $(T^*_\phi \otimes I_{\mathbb{C}^{p+m}})$-invariant subspace $\mathcal{K}$ of $H^2(\mathbb{D},\mathbb{C}^{p+m})$ such that
\begin{eqnarray*}
\mathcal{M}=[\mathcal{G}, I_m]\mathcal{K},
\end{eqnarray*}
where $\mathcal{G}$ is the matrix with $p$ columns that form an orthonormal basis of the subspace 
$span\{P_\mathcal{M} T_\Phi F_i: i=1,2,\dots,n\}$, and $I_m$ is the identity operator on $H^2(\bb D, \bb C^m)$. Furthermore, for each $F = \mathcal{G}R + H \in \mathcal{M}$,
\begin{eqnarray*}
    \|F\|^2=\|R\|^2 + \|H\|^2, \text{ where } (R,H)\in \mathcal{K}.
\end{eqnarray*}

Conversely, suppose a subspace of $\mathcal{M}$ of $H^2(\mathbb{D}, \mathbb{C}^m)$ has the representation $\mathcal{M}=[\mathcal{G},I_m]\mathcal{K}$, where $\mathcal{G}$ is the matrix with 
columns $\{F_1, \dots, F_p\}$, an orthonormal set in $\mathcal{M}$, and 
$\mathcal{K}\subset \{(\tilde{R}\circ \phi, H): \tilde{R}\in H^2(\mathbb{D}, \mathbb{C}^p), \ H\in H^2(\mathbb{D}, \mathbb{C}^m)\}$ is a $(T^*_\phi \otimes I_{\mathbb{C}^{p+m}})$-invariant subspace, along with a unitary from $\mathcal{K}$ onto $\mathcal{M}$ given by 
$
\begin{pmatrix}
R\\H
\end{pmatrix}
\mapsto \mathcal{G}R+H$. Then $\mathcal{M}$ is almost $T^*_\Phi$-invariant with defect at most $p$.  
\end{theorem}

\noindent {\bf Application III: Nearly invariant subspaces.}

\begin{definition} Suppose $\phi$ and $\psi$ are two inner functions in $H^2(\mathbb{D})$. Then a non-zero subspace $\mathcal{M}\subset H^2(\mathbb{D},\mathbb{C}^m)$ is called nearly $T^*_{\phi, \psi}$-invariant if   
 \begin{eqnarray*}
  T^*_\Phi(\mathcal{M}\cap T_\Psi H^2(\mathbb{D},\mathbb{C}^m))\subset \mathcal{M},
\end{eqnarray*}
where $T_\Phi:=T_{\phi} \otimes I_{\mathbb{C}^m}$, $T_{\Psi}:=T_{\psi} \otimes I_{\mathbb{C}^m}$.
\end{definition}

For $\phi=\psi=z$ and $m=1$, a nearly $T_{\phi,\psi}^*$-invariant subspace is simply a nearly invariant subspace under the backward shift $T_z^*.$

\begin{theorem}\label{Th3.13}
Suppose $B$ and $B^{'}$ are two finite Blaschke products such that $B$ vanishes at zero, and $B$ divides $B^{'}$. Then a non-zero subspace $\mathcal{M}\subset H^2(\mathbb{D},\mathbb{C}^m)$ is nearly $T^*_{B,B^{'}}$-invariant if and only if there exists a $(T^*_B \otimes I_{\mathbb{C}^p})$-invariant subspace $\mathcal{N}$ of $H^2(\mathbb{D},\mathbb{C}^p)$ such that
\begin{eqnarray*}
    \mathcal{M}=\mathcal{G}\mathcal{N},
\end{eqnarray*}
where $\mathcal{G}$ is the matrix with $p$ columns that form an orthonormal basis of $\mathcal{M}\ominus (\mathcal{M}\cap T_\mathcal{B^{'}} H^2(\mathbb{D},\mathbb{C}^m))$ with  $T_{\mathcal{B}^{'}}:=T_{B^{'}} \otimes I_{\mathbb{C}^m}$, along with a unitary $\cl U: \mathcal{N} \longrightarrow \mathcal{M}$ defined by
\begin{eqnarray*}
  \cl U (R)=\mathcal{G}R.
\end{eqnarray*}
\end{theorem}

\noindent Lastly, we study the notion of nearly $T^*_{\phi, \psi}$-invariant subspaces of $H^2(\bb D, \bb C^m)$ with finite defect $n$. 

\begin{definition}\label{Def4.1} Let $\phi$ and $\psi$ be two inner functions in $H^2(\mathbb{D})$. Then a non-zero subspace $\mathcal{M}$ of $H^2(\mathbb{D},\mathbb{C}^m)$ is said to be a nearly $T^*_{\phi, \psi}$-invariant with finite defect $n$ if there exists an $n$-dimensional subspace $\mathcal{F}$ (orthogonal to $\mathcal{M}$) such that 
\begin{equation}\label{innerdefect}
T^*_{\Phi}\Big(\mathcal{M}\cap T_{\Psi} H^2(\mathbb{D},\mathbb{C}^m)\Big) \subseteq \mathcal{M} \oplus \mathcal{F},
\end{equation} 
where $T_{\Phi}:= T_{\phi}\otimes I_{\mathbb{C}^m}$ and $T_{\Psi}:= T_{\psi}\otimes I_{\mathbb{C}^m}$.
\end{definition}

Similar to the almost invariant case, here also we have a notion of the defect and the defect space. A finite dimensional subspace of the smallest dimension that satisfies (\ref{innerdefect}) is called the \textit{defect space} of $\mathcal{M}$, and its dimension is called the \textit{defect} of $\mathcal{M}$. We shall not make any distinction between the terminologies, as the difference will always be apparent from the context. 

Note that the notion of a nearly $T^*_{\phi, \psi}$-invariant subspace is a particular case of a nearly $T^*_{\phi, \psi}$-invariant subspace with finite defect. In the following result, given two finite Blaschke products $B$ and $B^{'}$, we describe nearly $T^*_{B, B^{'}}$-invariant subspace of $H^2(\bb D, \bb C^m)$ with defect $1$. For notational clarity and convenience, here we are restraining ourselves to the simplest case of defect one to highlight the key features of the result rather than the tedious details; the results for the general case are similar and we give them in Section \ref{Sec4}.  

\begin{theorem}\label{Th4.2}
Suppose $B$ and $B^{'}$ are two finite Blaschke products such that $B$ vanishes at zero, and $B$ divides $B^{'}$ . Let $\mathcal{M}$ be a non-zero nearly $T^*_{B,B^{'}}$-invariant subspace of $H^2(\mathbb{D},\mathbb{C}^m)$ with defect $1$, and let $J$ be a function of unit norm in the defect space. 
\begin{enumerate}
\item[(i)] If $\mathcal{M}\not\subset T_{\mathcal{B^{'}}} H^2(\mathbb{D},\mathbb{C}^m)$, then there exists a $T^*_B\otimes I_{\mathbb{C}^{p+1}}$-invariant subspace $\mathcal{K}$ of $H^2(\mathbb{D},\mathbb{C}^{p+1})$ such that
\begin{eqnarray*}
    \mathcal{M}=\Big\{F\in H^2(\bb D, \bb C^m): F=\mathcal{G}R + BhJ \ \ {\rm with} \ \  (R,h)\in \mathcal{K}\Big\},
\end{eqnarray*}
where $\mathcal{G}$ is the matrix with $p$ columns that form an orthonormal basis of $\mathcal{M}\ominus (\mathcal{M}\cap T_{\mathcal{B^{'}}} H^2(\mathbb{D},\mathbb{C}^m))$. Further, for 
$F=\mathcal{G}R+ BhJ\in \mathcal{M}$,
\begin{eqnarray*}
\|F\|^2=\|R\|^2 + \|h\|^2.
\end{eqnarray*}

\item[(ii)] If $\mathcal{M}\subset T_{\mathcal{B^{'}}} H^2(\mathbb{D},\mathbb{C}^m)$, then there exists a 
$T^*_B$-invariant subspace $\mathcal{K}$ of $H^2(\mathbb{D})$ such that
\begin{eqnarray*}
    \mathcal{M}=\Big\{F\in H^2(\bb D, \bb C^m): F=BhJ \ \ {\rm with} \ \ h\in \mathcal{K} \Big\}.
\end{eqnarray*}
Further, for $F=BhJ\in \mathcal{M}$,
\begin{eqnarray*}
\|F\|=\|h\|.
\end{eqnarray*}
\end{enumerate}
\end{theorem}

\noindent The following is a converse of Theorem \ref{Th4.2}. 

\begin{theorem}\label{Th4.4} 
 Suppose $B$ and $B^{'}$ are two finite Blaschke products such that $B$ divides $B^{'}$, and $B$ vanishes at zero. Let $\mathcal{M}\subset H^2(\mathbb{D},\mathbb{C}^m)$ has either of  the following two representations
\begin{enumerate}
\item[(i)] $$
\mathcal{M}=\Big\{F\in H^2(\bb D, \bb C^m): F=\mathcal{G}R + BhJ \ \ {\rm for} \ (R, h)\in \mathcal{K}\Big\},
$$ 
where $J\in H^2(\bb D, \bb C^m)$ is of unit norm, $\mathcal{G}$ is the matrix with $p$ columns that form an orthonormal basis of $\mathcal{M}\ominus (\mathcal{M}\cap T_{\mathcal{B^{'}}} H^2(\mathbb{D},\mathbb{C}^m))$,  $\mathcal{K}$ is a $T^*_B \otimes I_{\mathbb{C}^{p+1}}$-invariant subspace of $H^2(\mathbb{D},\mathbb{C}^{p+1})$ with $\cl K\subseteq \{(\tilde{R}\circ B, \tilde{h}\circ B): \tilde{R}\in H^2(\bb D, \bb C^p), \ \tilde{h}\in H^2(\bb D)\}$, along with a unitary from $\mathcal{K}$ onto $\mathcal{M}$ defined by $(R,h)\mapsto \mathcal{G}R+BhJ$.

\item[(ii)] 
$$
\mathcal{M}=\Big\{F\in H^2(\bb D, \bb C^m): F=BhJ \ \ {\rm for} \ h\in \mathcal{K}\Big\},
$$ 
where $J\in H^2(\bb D, \bb C^m)$ is of unit norm and $\mathcal{K}$ is a $T^*_B$-invariant subspace of $H^2(\bb D)$ with $\cl K\subseteq \{\tilde{h}\circ B: \tilde{h}\in H^2(\bb D)\}$, along with a unitary from $\mathcal{K}$ onto $\mathcal{M}$ defined by $h\mapsto BhJ$.
 \end{enumerate}
 Then $\mathcal{M}$ is nearly $T^*_{B,B^{'}}$-invariant with defect 1 and $J$ spans the defect space.
\end{theorem}

\section{Proofs of Theorems \ref{Th3.2} \& \ref{Th3.6}, and their Applications I, II \& III}\label{Sec3} 
Before we give the proof of Theorem \ref{Th3.2}, first we state the following result from \cite{DS}. It describes invariant subspaces of finite-rank perturbations of a Toeplitz operator $T^*_\phi$ on $H^2(\mathbb{D})$ with the associated symbol $\phi$ an inner function. 

\begin{theorem}\cite[Theorem 2.1]{DS}\label{Th3.1}
Suppose  $\phi \in H^2(\mathbb{D})$ is an inner function that vanishes at zero; $\{u_i\}_{i=1}^{k}$ 
and $\{v_i\}_{i=1}^{k}$ are two orthonormal  sets in $H^2(\mathbb{D})$. Let $\mathcal{M}$ be a non-zero subspace of $H^2(\mathbb{D})$ that is invariant under $T^*_\phi - \sum_{i=1}^{k}v_i \otimes u_i$. Then there exists a $(T^*_\phi \otimes I_{\mathbb{C}^{p+1}})$-invariant subspace $\mathcal{K}$ of $H^2(\mathbb{D},\mathbb{C}^{p+1})$ such that
$$
    \mathcal{M}=[\mathcal{G}, 1]\mathcal{K},
$$
where $\mathcal{G}=[g_1, \dots, g_p]$ with $\{g_1, \dots, g_p\}$ an orthonormal basis of $span\{P_{\mathcal{M}}u_i : 1 \le i \le k\}$.
Moreover, $\|f\|^2=\|F\|^2+ \|h\|^2$ for all $f=\mathcal{G}F+ h\in \mathcal{M}.$
\end{theorem}    

This result provides a precise structure of invariant subspaces under finite-rank perturbations 
of the Toeplitz operator $T^*_\phi$ on $H^2(\mathbb{D})$. Note that our result, Theorem \ref{Th3.2}, extends this description to the vector-valued setting. This extension not only broadens the scope of Theorem \ref{Th3.1} but also brings to light the technical challenges that arise due to the added complexity of the vector-valued context. Before we proceed with our proof, we note that the assumption of orthonormality of the finite sets $\{u_i\}_{i=1}^{k}$ and $\{v_i\}_{i=1}^{k}$ in Theorem \ref{Th3.1} is redundant; therefore, we do not impose it in our results. 

\vspace{.3 cm}
 
\noindent{\underline{\bf Proof of Theorem \ref{Th3.2}}}  
 
 \vspace{.2 cm} 
 
Suppose $span\{P_{\mathcal{M}}U_i : 1 \le i \le k\}$ is denoted by $\mathcal{W}$. First, we consider the case when $\mathcal{W}=\{0\}$. Then $\{U_1, U_2, \dots, U_k\}$ is orthogonal to $\mathcal{M}$. Therefore, for any $F \in \mathcal{M}$,
\begin{eqnarray*}
   \Big(T^*_\Phi - \sum_{i=1}^{k}V_i \otimes U_i\Big)F &=& T^*_\Phi F - \sum_{i=1}^{k} \langle F,U_i \rangle V_i\\
   &=& T^*_\Phi F.
\end{eqnarray*}
This shows that $\mathcal{M}$ is $T^*_\Phi$-invariant. Thus, we can write 
\begin{eqnarray*}
    \mathcal{M}=[I_m]\mathcal{K}, \text{ where } \mathcal{K}=\mathcal{M},
\end{eqnarray*}
which proves the result for this case.  

Now, we assume that $\mathcal{W}\ne \{0\}$. Let the dimension of $\cl W$ equals $p$, and let 
$\{G_1, \dots, G_p\}$ be its orthonormal basis. We decompose $\mathcal{M}$ as 
\begin{eqnarray*}
    \mathcal{M}= \mathcal{W} \oplus (\mathcal{M} \ominus \mathcal{W}).
\end{eqnarray*}
Therefore, any $F \in \mathcal{M}$ can be expressed as
\begin{eqnarray}\label{e3.1}
F=a_{01}G_1 +...+ a_{0p}G_p +  P_{\mathcal{M} \ominus \mathcal{W}}(F), \ {\rm where} \ a_{0i}\in \mathbb{C}.
\end{eqnarray}
Now, $F_1:=P_{\mathcal{M} \ominus \mathcal{W}}(F)$ is in $\mathcal{M}$, which implies   
\begin{eqnarray*}
L_1 :=\Big(T^*_\Phi - \sum_{i=1}^{k} V_i \otimes U_i\Big)F_1 =T^*_\Phi F_1 \in \mathcal{M},
\end{eqnarray*}
which further yields 
\begin{eqnarray*}
    F_1=T_\Phi L_1 + P_{\mathcal{K}_\Phi}F_1,
\end{eqnarray*}
where $\mathcal{K}_\Phi=H^2(\mathbb{D},\mathbb{C}^m)\ominus T_\Phi H^2(\mathbb{D},\mathbb{C}^m)$. 
Then, using Equation (\ref{e3.1}), we get
\begin{eqnarray*}
    F=(a_{01}G_1 +...+ a_{0p}G_p) + T_\Phi L_1 + P_{\mathcal{K}_\Phi}F_1
\end{eqnarray*}
and
\begin{equation}\label{e3.2n}
    \|F\|^2=\sum_{i=1}^{p} |a_{0i}|^2 + \|L_1\|^2 +\|P_{\mathcal{K}_\Phi}F_1\|^2.
\end{equation}
In other words,
\begin{equation}\label{e3.2}
F=\mathcal{G}A_0 + T_\Phi L_1 + P_{\mathcal{K}_\Phi}F_1, 
\end{equation}
where $\cl G=[G_1, \dots, G_p]$, that is, the matrix with $p$ columns $G_1, \dots, G_p$.
Further, since $L_1 \in \mathcal{M}$, therefore by applying the similar arguments as employed above, we obtain 
\begin{eqnarray}\label{e3.3}
 L_1=\mathcal{G}A_1 + T_\Phi L_2 + P_{\mathcal{K}_\Phi}F_2,  \end{eqnarray}
where $L_2=(T^*_\Phi P_{\mathcal{M} \ominus \mathcal{W}})L_1$,  $F_2=P_{\mathcal{M} \ominus \mathcal{W}}(L_1)$, and $A_1\in \mathbb{C}^p$.  

Then, using Equations (\ref{e3.3}) and (\ref{e3.2}), we get

\begin{eqnarray*}
F &=& \mathcal{G}A_0 + T_\Phi(\mathcal{G}A_1 +  T_\Phi L_2 + P_{\mathcal{K}_\Phi}F_2) + P_{\mathcal{K}_\Phi}F_1\\
&=& \mathcal{G}(A_0 + A_1 \phi) + L_2 \phi^2 + P_{\mathcal{K}_\Phi}F_1 + (P_{\mathcal{K}_\Phi}F_2)\phi.
\end{eqnarray*}
and 
$$
    \|F\|^2=\|A_0\|^2 + \|A_1\|^2 + \|L_2\|^2+\|P_{\mathcal{K}_\Phi}F_1 \|^2 + \|P_{\mathcal{K}_\Phi}F_2\|^2.
$$

Continuing the same process, we get, for each $n \in \mathbb{N}$
\begin{equation}\label{MTe1}
F = \mathcal{G}\Bigg(\sum_{i=0}^{n-1}A_i\phi^i\Bigg) + L_{n}\phi^{n} + \sum_{i=1}^{n}(P_{\mathcal{K}_\Phi}F_i)\phi^{i-1}
\end{equation}
and  
\begin{equation}\label{MTe2}
    \|F\|^2=\sum_{i=0}^{n-1}\|A_i\|^2 + \|L_{n}\|^2 + \sum_{i=1}^{n}\|P_{\mathcal{K}_\Phi}F_i\|^2.
\end{equation}

Let $P_\mathcal{W}$ denote the orthogonal projection of $H^2(\mathbb{D},\mathbb{C}^m)$ onto $\mathcal{W}$. Then for any $L \in \mathcal{M}$, $P_{(\mathcal{M}\ominus \mathcal{W})}L=(I_m -P_{\mathcal{W}})L$. Also it is easy to observe that $T_\Phi $ is a  $C._{0}$ contraction, therefore by Lemma \ref{Lem2.2},  $T_\Phi(I_m-P_{\mathcal{W}})$ is a $C._{0}$ contraction.
Thus, 
\begin{eqnarray*}
\|L_{n}\|&=& \|\big(T^*_\Phi P_{\mathcal{M}\ominus\mathcal{W}}\big)^{n}F\|\\
&=& \|\big(T^*_\Phi (I_m-P_{\cl W})\big)^{n}F\|\\
&=& \|T^*_\Phi\Big((I_m-P_{\cl W})T^*_\Phi\Big)^{n-1}(I_m-P_{\mathcal{W}})F\|\\
&\le & \|T^*_\Phi\|\|\big((T_\Phi(I_m-P_\mathcal{W}))^*\big)^{(n-1)}(I_m-P_{\cl W})F\|\\
& \longrightarrow & 0 \text{ as } n \longrightarrow \infty.
\end{eqnarray*}
This shows that $L_{n} \longrightarrow 0 \text{ as } n \longrightarrow \infty$.  

Next, using Equation (\ref{MTe2}), $\{A_i\}_{i=0}^\infty$ is a square summable sequence in $\bb C^p$.  In addition, $\{\phi^i\}_{i=0}^\infty$ is an orthonormal sequence in $H^2(\bb D)$. Therefore, 
$R:=\sum_{i=0}^{\infty}  A_i \phi^i$ is a well-defined function in $H^2(\bb D, \bb C^p)$. Further, $T_\Phi$ is a pure isometry on $H^2(\bb D, \bb C^m);$ therefore, by Wold decomposition 
$$
H^2(\mathbb{D},\mathbb{C}^m)= \bigoplus_{i=0}^{\infty} T_\Phi^i\mathcal{K}_\Phi. 
$$

Then, since $\sum_{i=1}^{\infty}\|P_{\mathcal{K}_{\Phi}}F_{i}\|^2 < \infty$, therefore $\sum_{i=1}^{\infty}T_{\Phi}^{i-1}(P_{\mathcal{K}_{\Phi}}F_i)$ is a well-defined function in 
$H^2(\bb D, \bb C^m)$. Note that 
$$
T^{i-1}_\Phi(P_{\mathcal{K}_{\Phi}}F_i)=(P_{\mathcal{K}_{\Phi}}F_i)\phi^{i-1}.
$$
Thus, we conclude that $H:=\sum_{n=1}^{\infty}(P_{\mathcal{K}_{\Phi}}F_{n})\phi^{n-1}$ is a well-defined function in $H^2(\bb D, \bb C^m).$ We now consider the following $n^{th}$ partial sums:
\begin{eqnarray*}
R_n=\sum_{i=0}^{n}A_i\phi^i \quad \text{ and } \quad H_n=\sum_{i=1}^{n}(P_{\mathcal{K}_{\Phi}}F_{i})\phi^{i-1}.
\end{eqnarray*}

Since $R_n \longrightarrow R$ and $H_n \longrightarrow H$, therefore, for each $z\in \bb D$, $\mathcal{G}(z)R_{n-1}(z) \longrightarrow \mathcal{G}(z)R(z)$ and $H_n(z) \longrightarrow H(z)$. Thus, $\mathcal{G}(z)R_{n-1}(z) + H_n(z) \longrightarrow \mathcal{G}(z)R(z) + H(z)$ for each $z\in \bb D$. 
But,
\begin{eqnarray*}
    \|F-(\mathcal{G}R_{n-1} + H_n)\|=\|L_{n}\| \longrightarrow 0,
\end{eqnarray*}
which implies that, for each $z\in \bb D, \ \ \mathcal{G}(z)R_{n-1}(z) + H_n(z) \longrightarrow F(z)$. Hence, we conclude 
\begin{eqnarray}\label{e3.5}
F = \mathcal{G}R + H.
\end{eqnarray}
Additionally, Equation (\ref{MTe2}) implies  
\begin{eqnarray}
    \|F\|^2=\|R\|^2 +\|H\|^2 .
\end{eqnarray}

We note that $F$ in Equation (\ref{e3.5}) satisfies the following conditions:  
\begin{enumerate}
\item[C1.] $R=\sum_{n=0}^{\infty}  A_n \phi^n$, $P_\mathcal{W}L_n=\mathcal{G}A_n$, where 
$L_n=(T^*_\Phi P_{\mathcal{M}\ominus \mathcal{W}})^nF \ \forall \ n \ge 0 $.
\item[C2.] $H=\sum_{n=1}^{\infty} (P_{\mathcal{K}_{\Phi}}F_{n})\phi^{n-1},$ where 
$F_n=P_{\mathcal{M}\ominus \mathcal{W}}L_{n-1} \ \forall \ n\ge 1.$
\item[C3.] $\|F\|^2=\|R\|^2 +\|H\|^2.$
\end{enumerate}

The representation of $F$ in Equation (\ref{e3.5}) that satisfies conditions C1, C2, and C3 is unique. Let $F$ has the following two representations satisfying C1, C2, and C3: 
$$
F = \mathcal{G}R+ H, \ {\rm where} \ R=\sum_{n=0}^{\infty}  A_n \phi^n \ {\rm and} \ 
H=\sum_{n=1}^{\infty}  (P_{\mathcal{K}_{\Phi}}F_{n})\phi^{n-1},
$$
and 
$$
F=\mathcal{G}\tilde{R}+ \tilde{H}, \text{ where }\tilde{R}=\sum_{n=0}^{\infty}  \tilde{A_n} \phi^n \text{ and }\tilde{H}=\sum_{n=1}^{\infty}  (P_{\mathcal{K}_{\Phi}}\tilde{F}_n)\phi^{n-1}.
$$
Then from C1, we have
\begin{eqnarray*}
    \mathcal{G}A_n=P_\mathcal{W}(T^*_\Phi P_{\mathcal{M}\ominus \mathcal{W}})^nF=\mathcal{G}\tilde{A_n},
\end{eqnarray*}
which shows that $A_n=\tilde{A_n}$. Thus $F=\tilde{F}$; hence, $H=\tilde{H}$. Therefore, the representation of $F$ given by Equation (\ref{e3.5}) is unique.  

\vspace{.2 cm}

We now define a set $\mathcal{K} \subset H^2(\mathbb{D},\mathbb{C}^{p+m})$ as:
$$
\mathcal{K}=\{(R,H): F = \mathcal{G}R + H \in \mathcal{M} \ {\rm and} \ F, R, \ {\rm and} \ H \ {\rm satisfy \ C1, C2, C3} \}.   
$$

First, we show that $\mathcal{K}$ is a vector subspace of $H^2(\bb D, \bb C^{p+m})$. Take $(R,H)$ and $(\tilde{R}, \tilde{H})$ in $\mathcal{K}$, and a scalar $\alpha\in \bb C$. Then there exist $F$ and $\tilde{F}\in \mathcal{M}$ such that
$$
F=\mathcal{G}R+ H, \ \text{ where} \ R=\sum_{n=0}^{\infty}  A_n \phi^n, \ H=\sum_{n=1}^{\infty}  (P_{\mathcal{K}_{\Phi}}F_{n})\phi^{n-1},
$$
and 
$$
\tilde{F} = \mathcal{G}\tilde{R} + \tilde{H}, \ \text{where} \ \tilde{R} = \sum_{n=0}^{\infty}  \tilde{A_n} \phi^n, \ \tilde{H}=\sum_{n=1}^{\infty}  (P_{\mathcal{K}_{\Phi}}\tilde{F}_n)\phi^{n-1}. 
$$

This gives 
\begin{eqnarray}\label{e3.7}
F+\alpha \tilde{F} &=& \mathcal{G}(R+ \alpha\tilde{R})+ (H+\alpha\tilde{H})\nonumber \\
&=&\mathcal{G}\Bigg(\sum_{n=0}^{\infty}  (A_n+ \alpha\tilde{A_n}) \phi^n \Bigg)+ \sum_{n=1}^{\infty}  P_{\mathcal{K}_{\Phi}}(F_{n}+\alpha\tilde{F_n})\phi^{n-1},
\end{eqnarray}
and
\begin{eqnarray}\label{e3.8}
    \mathcal{G}(A_n + \alpha\tilde{A_n})&=& \mathcal{G}A_n +\alpha \mathcal{G} \tilde{A_n} \nonumber\\
    &=& P_\mathcal{W}(T^*_\Phi P_{\mathcal{M}\ominus \mathcal{W}})^nF + \alpha P_\mathcal{W}(T^*_\Phi P_{\mathcal{M}\ominus \mathcal{W}})^n\tilde{F} \nonumber \\
    &=& P_\mathcal{W}(T^*_\Phi P_{\mathcal{M}\ominus \mathcal{W}})^n(F+\alpha\tilde{F}).
\end{eqnarray}

Equations (\ref{e3.7}) and (\ref{e3.8}) also yield 
\begin{eqnarray}\label{e3.9}
    \|F+\alpha\tilde{F}\|^2=\|R+\alpha\tilde{R}\|^2+\|H+\alpha\tilde{H}\|^2 .
\end{eqnarray}

Thus, from Equations (\ref{e3.7}), (\ref{e3.8}), and (\ref{e3.9}), it is evident that $F+\tilde{F},$ 
$R+\tilde{R},$ and $H+\tilde{H}$ satisfy C1, C2, and C3. Consequently, $\mathcal{K}$ is a vector subspace. To show that $\mathcal{K}$ is closed, let $\{(R_n,H_n)\}_n$ be a Cauchy sequence in $\mathcal{K}$. Then 
\begin{eqnarray*}
\|(R_n,H_n)-(R_r,H_r)\|^2&=& \|R_n - R_r\|^2 + \|H_n-H_r\|^2\\
&=& \|\mathcal{G}(R_n-R_r) + (H_n-H_r)\|^2.
\end{eqnarray*}

This shows that the sequence $\{F_n : F_n=\mathcal{G}R_n + H_n\}_n$ is a Cauchy sequence in $\mathcal{M}$; therefore, there exist $F\in \mathcal{M}$ and $(R,H)\in \mathcal{K}$ such that 
$F_n \longrightarrow F=\mathcal{G}R+H$. Thus, 
\begin{eqnarray*}
    \|(R_n, H_n) - (R, H)\|^2=\|F_n - F\|^2 \longrightarrow 0,
\end{eqnarray*}
which shows that $(R_n, H_n) \longrightarrow (R,H)$, and thus $\mathcal{K}$ is closed.  

Lastly, we show that $\mathcal{K}$ is $( T^*_\phi \otimes I_{\mathbb{C}^{p+m}})$-invariant. Choose $(R,H)\in \mathcal{K}$, then by definition of $\mathcal{K}$, there exists $F\in \mathcal{M}$ such that
$$
F=\mathcal{G}R + H,
$$
and $F, \ R,$ and $H$ satisfy conditions C1, C2, and C3. Then, by the uniqueness of the representation of $F$, we have 
$$
F=\mathcal{G}A_0 + L_1\phi + P_{\mathcal{K}_{\Phi}}F_1,
$$
where $L_1=\mathcal{G}\Big(\sum_{n=1}^{\infty}A_n\phi^{n-1}\Big) + \sum_{n=2}^{\infty}(P_{\mathcal{K}_{\Phi}}F_n)\phi^{n-2}\in \cl M$ with $R=\sum_{n=0}^\infty A_n \phi^n$ and 
$H=\sum_{n=1}^\infty(P_{{\cl K}_\Phi}F_n) \phi^{n-1}.$ Furthermore, 
\begin{eqnarray*}
T_\phi^*\otimes I_{\bb C^{p+m}}(R, H) &=& \Big((T_{\phi}^*\otimes I_{\bb C^p})R, (T_{\phi}^*\otimes I_{\bb C^m}H)\Big)\\
&=& \Big(\sum_{n=1}^\infty A_n \phi^{n-1}, \sum_{n=2}^{\infty}(P_{\mathcal{K}_{\Phi}}F_n)\phi^{n-2}\Big).
\end{eqnarray*}
But, $L_1=\mathcal{G}\Big(\sum_{n=1}^{\infty}A_n\phi^{n-1}\Big) + \sum_{n=2}^{\infty}(P_{\mathcal{K}_{\Phi}}F_n)\phi^{n-2}\in \cl M$; hence we conclude $T_\phi^*\otimes I_{\bb C^{p+m}}(R, H)\in \cl K.$ This completes the proof.  
\qed

Our proof of Theorem \ref{Th3.2} does not simply give the existence of the subspace $\mathcal{K}$  
that appears in the statement of the theorem; in fact, it
constructs it. Before proceeding further, we want to bring to light some crucial details about $\mathcal{K}$ hidden in the proof. Apart from being informative in their own right, these are significant for our work as we advance in this paper; therefore, we note these in the following remark. 
\begin{remark}\label{uni3}

We have shown in the proof of Theorem \ref{Th3.2} that given an element $F$ in $\mathcal{M},$ the functions $R\in H^2(\mathbb{D}, \mathbb{C}^p)$ and $H\in H^2(\mathbb{D}, \mathbb{C}^m)$ such that $(R,H)\in \mathcal{K}$ with $F=\mathcal{G}R+H$ indeed are uniquely determined by conditions C1, C2, and C3 (as stated in the proof of Theorem \ref{Th3.2}). Hence, we precisely know the subspace $\mathcal{K}$. Furthermore, this allows us to define a unitary $\cl U: \mathcal{K}\to \mathcal{M}$ by  
\begin{equation}\label{rep3}
\cl U = [\mathcal{G}, I_{m}]:
\begin{pmatrix}
 R\\
 H
\end{pmatrix}
\mapsto \mathcal{G}R+H.
\end{equation}
\end{remark}

Interestingly, the existence of such a unitary is sufficient to guarantee that elements of $\mathcal{K}$ are determined by conditions C1, C2, and C3. We make it precise in the following result.

\begin{proposition}\label{uni1} Let $\mathcal{M}$ be a subspace of $H^2(\mathbb{D}, \mathbb{C}^m)$, $\{G_i\}_{i=1}^{p}$ be an orthonormal set in 
$\mathcal{M}, \ \mathcal{K}$ be a subspace of $H^2(\mathbb{D}, \mathbb{C}^{p+m})$, and let 
$\cl U:\mathcal{K} \to \mathcal{M}$ be the unitary defined by (\ref{rep3}), where $\mathcal{G}$ 
is the matrix consisting of columns $G_1, \dots, G_p$. Further, let  $\phi$ be an inner function in $H^2(\mathbb{D})$ that vanishes at zero such that $\mathcal{K}$ is invariant under $T_\phi^*\otimes I_{{\mathbb{C}}^{p+m}}$, and $\mathcal{K}\subseteq \{(\tilde{R}\circ \phi, H): \tilde{R}\in H^2(\mathbb{D}, \mathbb{C}^p), \ H\in H^2(\mathbb{D}, \mathbb{C}^m)\}$. Then given any $(R, H)\in \mathcal{K}$, the functions $F, R,$ and $H$ satisfy conditions C1, C2, and C3, where $F=\mathcal{G}R+H \in \mathcal{M}$, and $\mathcal{W}$ (as appears in conditions C1, C2, and C3) equals the linear span of $\{G_i\}_{i=1}^{p}$. 
\end{proposition}

\begin{proof} Let $(R, H)\in \mathcal{K}.$ Then $F=\mathcal{G}R+H \in \mathcal{M}.$ Let 
$T_\Phi := T_\phi\otimes I_{\bb C^m}.$ First, we recall conditions C1, C2, and C3 that we want to prove. 
\begin{enumerate}
\item[C1.] $R=\sum_{n=0}^{\infty}  A_n \phi^n, \ \mathcal{G}A_n=P_\mathcal{W}L_n$, where 
$L_n=(T^*_\Phi P_{\mathcal{M}\ominus \mathcal{W}})^n F\ \forall \ n \ge 0$.
\item[C2.] $ H=\sum_{n=0}^{\infty} P_{\mathcal{K}_{\Phi}}(H_{n})\phi^{n},$ where 
$H_n=P_{\mathcal{M}\ominus \mathcal{W}}L_{n} \ \forall \ n\ge 0.$
\item[C3.] $\|F\|^2=\|R\|^2 +\|H\|^2.$
\end{enumerate}

Since $\cl U$ is a unitary, the condition C3 holds. For conditions C1 and C2, we first note that our hypotheses allow us to assume $R=\sum_{n=0}^{\infty}A_n\phi^n$ for some square summable sequence 
$\{A_n\}_{n=0}^\infty$ in $\mathbb{C}^p$. Also, since $T_\Phi$ is an isometry on $H^2(\mathbb{D}, \mathbb{C}^m)$, we can decompose $H$ as
$$H=\bigoplus_{n=0}^\infty T_\Phi^{n}(P_{\mathcal{K}_\Phi}H_n)=\bigoplus_{n=0}^\infty P_{\mathcal{K}_\Phi}(H_n) \phi^{n},$$ 
where $\cl K_{\Phi}=H^2(\bb D, \bb C^m)\ominus T_{\Phi} H^2(\bb D, \bb C^m)$ and $H_n\in H^2(\mathbb{D}, \mathbb{C}^m)$. 

All that remains to complete the proof is to show that, for each $n\ge 0$,  
$$
\mathcal{G}A_n = P_\mathcal{W}L_n \ \ {\rm and} \ \  P_{\mathcal{K}_\Phi}H_n=P_{\mathcal{K}_\Phi}P_{\mathcal{M}\ominus \mathcal{W}}L_{n},
$$ 
where $L_n=(T^*_\Phi P_{\mathcal{M}\ominus \mathcal{W}})^nF$. We start by decomposing $F$ as  
\begin{eqnarray*}
F &=&\mathcal{G}R(0)+\mathcal{G}(R-R(0))+H\\
&=& \mathcal{G}R(0)+\mathcal{G}(R-R(0))+T_\Phi T_\Phi^* H+P_{\mathcal{K}_\Phi}H\\
&=& \mathcal{G}R(0)+T_\Phi (\mathcal{G}T_\Psi^*R+T_\Phi^* H)+P_{\mathcal{K}_\Phi}H,
\end{eqnarray*}
where $T_\Phi:=T_\phi\otimes I_{\mathbb{C}^{m}}$ and $T_\Psi := T_\phi \otimes I_{\mathbb{C}^p}$.

Since $\mathcal{K}$ is invariant under $T_\phi^*\otimes I_{{\mathbb{C}}^{p+m}}=T_\Psi^*\oplus T_\Phi^*$, therefore $(T_\Psi^*F, T_\Phi^* H)\in \mathcal{K}$. Then 
$Y_1:=\mathcal{G}T_\Psi^*F+T_\Phi^* H$ belongs to $\mathcal{M}.$ Thus, 
\begin{equation*}
F=\mathcal{G}R(0)+T_\Phi Y_1+P_{\mathcal{K}_\Phi}H.
\end{equation*}

Note that    
$$
\langle{F-\mathcal{G}R(0), G_i}\rangle= \langle{\mathcal{G}(R-R(0))+ H, G_i}\rangle 
= \langle{(R-R(0), H), (E_i, 0)}\rangle =0;
$$
therefore, $P_{\mathcal{W}}F=\mathcal{G}R(0)$ and $P_{\mathcal{M}\ominus \mathcal {W}}F=T_\Phi Y_1+P_{\mathcal{K}_\Phi}H$, which implies that 
$L_1=T_\Phi^*P_{\mathcal{M}\ominus \mathcal{W}}F=Y_1$ and $P_{\mathcal{K}_{\Phi}}P_{\mathcal{M}\ominus \mathcal{W}}F=P_{\mathcal{K}_\Phi}H.$ Also, $P_{\mathcal{K}_\Phi}H=P_{\mathcal{K}_\Phi}H_0.$ Thus, we conclude that 
$$
P_{\mathcal{W}}F=\mathcal{G}A_0, \ P_{\mathcal{K}_{\Phi}}P_{\mathcal{M}\ominus \mathcal{W}}F=P_{\mathcal{K}_\Phi}H_0. 
$$
 
Additionally, we have $L_1=Y_1\in \mathcal{M}$ and  
$$
L_1=\mathcal{G}T_\Psi^*R+T_\Phi^* H
$$
with $T_\Psi^*R = \sum_{n=0}^\infty A_{n+1}\phi^{n}$ and $T_\Phi^* H=\sum_{n=0}^\infty P_{\mathcal{K}_{\Phi}}(H_{n+1}) \phi^{n}.$ Then, following arguments similar to those used above, we first obtain  
$$
L_1=\mathcal{G}A_1+T_\Phi Y_2+P_{\mathcal{K}_\Phi}(T_\Phi^* H); 
$$
then deduce 
$$
P_{\mathcal{W}}L_1=\mathcal{G}A_1, \ P_{\mathcal{K}_{\Phi}}P_{\mathcal{M}\ominus \mathcal{W}}L_1=P_{\mathcal{K}_\Phi}H_1,
$$ 
and 
$$
L_2=(T_\Phi^*P_{\mathcal{M}\ominus \mathcal{W}})^2 F = T_\Phi^*P_{\mathcal{M}\ominus \mathcal{W}}L_1 = Y_2,
$$
where $$Y_2=\mathcal{G}T_\Psi^{*2}F+T_\Phi^{*2} H$$ with $T_\Psi^{*2}R = \sum_{n=0}^\infty A_{n+2}\phi^{n}$ and $T_\Phi^{*2} H=\sum_{n=0}^\infty P_{\mathcal{K}_{\Phi}}(H_{n+2}) \phi^n.$  

\vspace{.2 cm}

It is now straightforward to see that by continuing the similar steps, we can establish that, for each $n\ge 0$,
$$
P_{\mathcal{W}}L_n=\mathcal{G}A_n \ \ {\rm and}  \ \ P_{\mathcal{K}_\Phi}H_n = P_{\mathcal{K}_{\Phi}}P_{\mathcal{M}\ominus \mathcal{W}}L_n. 
$$
This completes the proof.
\end{proof}

\begin{remark}\label{uni} The outcomes of Remark \ref{uni3} and Proposition \ref{uni1} can be summarized together as follows: given a subspace $\mathcal{M}$ of 
$H^2(\mathbb{D}, \mathbb{C}^m)$, the representation of the subspace $\mathcal{K}$ of $H^2(\mathbb{D}, \mathbb{C}^{p+m})$ obtained in Theorem \ref{Th3.2} is equivalent to having a unitary between $\mathcal{M}$ and $\mathcal{K}$ given by (\ref{rep3}). 
\end{remark}

Further, we want to emphasize that establishing conditions C1, C2, and C3 for elements of $\mathcal{K}$, once we have the unitary, do not require the assumption that $\mathcal{M}$ is invariant under a perturbation of $T_{\Phi} ^*$. Indeed, it follows automatically as asserted in Theorem \ref{Th3.6}. Recall that it is a converse of Theorem \ref{Th3.2}. We note that it is a vector-valued analogue of the converse of Theorem \ref{Th3.1} proved in \cite[Theorem 3.1]{DS}. 

\vspace{.3cm}

\noindent \underline{{\bf Proof of Theorem \ref{Th3.6}}}

\vspace{.2 cm}

Suppose $F \in \mathcal{M}$, then there exists $(R,H)\in \mathcal{K}$ such that $ F=\mathcal{G}R + H$. Now, since $R= \tilde{R} \circ \phi$, for some $\tilde{R}\in H^2(\mathbb{D}, \mathbb{C}^p)$, and $T_\Phi$ is an isometry on $H^2(\mathbb{D}, \mathbb{C}^m)$, therefore we can express $F$ as 
\begin{eqnarray*}
F &=&\mathcal{G}R + H\\
&=&\mathcal{G}R(0) +\mathcal{G} (R-R(0)) +T_{\Phi}^*T_{\Phi}H + P_{\mathcal{K}_\Phi}H\\
&=&\mathcal{G}R(0) + T_{\Phi}(\mathcal{G}T^*_\Psi R +  T^*_\Phi H) + P_{\mathcal{K}_\Phi}H,
\end{eqnarray*}
where $T_\Psi:=T_\phi\otimes I_{\mathbb{C}^p}.$
\noindent This implies that
\begin{eqnarray}\label{e3.10}
    T^*_\Phi F= T^*_\phi (\mathcal{G}R(0)) + \mathcal{G}T^*_\Psi R +  T^*_\Phi H.
\end{eqnarray} 
Further,
\begin{eqnarray*}
 \Big(\sum_{i=1}^{p}T^*_\Phi G_i \otimes G_i\Big)F 
   &=&\sum_{i=1}^{p} \langle{F,G_i}\rangle T^*_\Phi G_i \\
   &=&\sum_{i=1}^{p} \langle \mathcal{G}R+H,\mathcal{G}E_i + 0 \rangle T^*_\Phi G_i\\
 &=& \sum_{i=1}^{p}\langle R, E_i \rangle T^*_\Phi G_i \nonumber \\ 
   &=& \sum_{i=1}^{p}\langle R(0), E_i \rangle T^*_\phi G_i  \\
   &=& T^*_\Phi \Big(\sum_{i=1}^{p} \langle R(0), E_i \rangle G_i \Big)\\
   &=& T^*_\Phi (\mathcal{G}R(0)).  
\end{eqnarray*}
Thus,  
\begin{equation} \label{e3.11}
  \Big(\sum_{i=1}^{p}T^*_\Phi G_i \otimes G_i\Big)F= T^*_\Phi (\mathcal{G}R(0)).
\end{equation}

\noindent Using Equations (\ref{e3.10}) and (\ref{e3.11}), we conclude 
\begin{eqnarray*}
    \Big(T^*_\Phi-\sum_{i=1}^{p}T^*_\Phi G_i \otimes G_i\Big)F &=& T^*_\Phi F - \Big(\sum_{i=1}^{p}T^*_\Phi G_i \otimes G_i\Big)F\\
    &=& T^*_\Phi (\mathcal{G}R(0))+ \mathcal{G}T^*_\Psi R +  T^*_\Phi H - T^*_\Phi (\mathcal{G}R(0)).\\
    &=& \mathcal{G}T^*_\Psi R +  T^*_\Phi H, 
\end{eqnarray*}
which belongs to $\mathcal{M}$, as $(R, H)\in \mathcal{K}$ and $\mathcal{K}$ is invariant under $T^*_\phi \otimes I_{\mathbb{C}^{p+m}}=T_{\Psi}^*\oplus T_{\Phi}^*$. This completes the proof.
\qed

\vspace{.2 cm}

We shall now present the proofs of Theorems \ref{Th3.7}, \ref{Th3.10} and \ref{Th3.13}. Recall that Theorem \ref{Th3.7} describes invariant subspaces of some specific finite-rank perturbations of certain Toeplitz operators on $H^2(\mathbb{D},\mathbb{C}^m)$. We note that it is a vector-valued extension of a characterization of invariant subspaces of certain finite-rank perturbations of $T_z$ on $H^2(\mathbb{D})$ given by Das and Sarkar in \cite[Theorems 4.1 \& 4.2]{DS}.

\vspace{.3 cm} 

\noindent\underline{{\bf Proof of Theorem \ref{Th3.7}}} 

\vspace{.2 cm} 

For the forward implication, we observe that $\mathcal{M}^{\perp}$ is invariant under  $T^*_\Phi-\sum_{i=1}^{k} U_i \otimes V_i$. Therefore, by Theorem \ref{Th3.2},  there exists a $(T^*_\phi \otimes I_{\mathbb{C}^{p+m}})$-invariant subspace $\mathcal{K}$ of $H^2(\mathbb{D},\mathbb{C}^{p+m})$ such that
\begin{eqnarray*}
    \mathcal{M}^{\perp}=[\mathcal{G}, I_m]\mathcal{K},
\end{eqnarray*}
along with a unitary  $\cl U: \mathcal{K} \longrightarrow \mathcal{M}^\perp$ given by 
 \begin{equation*}
      \cl U
      \begin{pmatrix}
      R\\
      H
      \end{pmatrix}=\mathcal{G}R+H,
 \end{equation*}
Set $\mathcal{N}:=\mathcal{K}^{\perp}$. Then
$\mathcal{N}$ is $(T_\phi \otimes I_{\mathbb{C}^{p+m}})$-invariant subspace of $H^2(\mathbb{D},\mathbb{C}^{p+m})$. Recall from the proof of Theorem \ref{Th3.2} that 
$\mathcal{N}^\perp=\mathcal{K}\subset \{(\tilde{R}\circ \phi, H): \tilde{R}\in H^2(\mathbb{D}, \mathbb{C}^p), \ H\in H^2(\mathbb{D}, \mathbb{C}^m)\}$. This completes the proof of the forward implication. The converse follows straightforwardly using Theorem \ref{Th3.6}.

For the moreover part, assume that $\{F_i\}_{i=1}^{p}\subset H^{\infty}(\mathbb{D},\mathbb{C}^m)$. Note that any $G \in \mathcal{M}^\perp$ can be expressed as
\begin{eqnarray*}
  G &=&\mathcal{G}R+H\\
    &=&\sum_{i=1}^{p}f_{i}F_i+ H,    
\end{eqnarray*}
where $(R,H)\in \mathcal{N}^\perp$ with $R=(f_1, \cdots, f_p), \ f_i\in H^2(\mathbb{D})$. Let 
$F \in \mathcal{M}$. Then, for $G\in \cl M^\perp$,  
\begin{eqnarray*}
0=\langle F,G \rangle&= &\langle F,\sum_{i=1}^{p}f_iF_i+ H \rangle \\
&=& \langle F,\sum_{i=1}^{p}f_iF_i \rangle + \langle F, H\rangle\\
&=& \sum_{i=1}^{p}\langle T^*_{F_i}F,f_i \rangle + \langle F, H \rangle\\
&=& \left\langle{ (T^*_{F_1}F,\dots,T^*_ {F_p}F,F), (R, H)}\right\rangle.
\end{eqnarray*}
Thus, we established that  
$$F\in \mathcal{M} \ \ {\rm if \ and \ only \ if} \ \ (T^*_{F_1}F,\dots,T^*_ {F_p}F,F) \in \mathcal{N}.$$ 
Hence, 
\begin{eqnarray*}
\mathcal{M}=\Big\{F \in H^2(\mathbb{D},\mathbb{C}^m): (T^*_{F_1}F,\dots, T^*_ {F_p}F,F) \in \mathcal{N}\Big\}.  
\end{eqnarray*}
\qed

\vspace{.2 cm}

We shall now prove Theorem \ref{Th3.10} that describes almost invariant subspaces of certain Toeplitz operators on the vector-valued Hardy spaces  $H^2(\mathbb{D},\mathbb{C}^m)$.  
 
The study and classification of almost invariant subspaces of the backward shift operator on the Hardy space $H^2(\mathbb{D})$ was initially carried out by Chalendar, Gallardo-Guti$\rm{\acute{e}}$rrez, and Partington in \cite[Corollary 3.4]{CGP}. Later, \cite[Corollary 3.6]{CDP1} extended this work to the vector-valued setting. Both these works, for their respective settings, first characterized nearly invariant subspaces with finite defect of the backward shift operator and then used it to deduce the representations of the almost invariant subspaces. 
Recently, Das and Sarkar \cite{DS} have taken a different route to derive their description of almost invariant subspaces of the backward shift operator on $H^2(\bb D).$ Indeed, they deduce their result (\cite[Corollary 5.2]{DS}) from Theorem \ref{Th3.1} by using Lemma \ref{Lem3.6} that establishes a connection of almost $T_\phi^*$-invariant subspaces with invariant subspaces of finite-rank perturbations of $T_\phi^*$.  

It is worth emphasizing that each of these approaches, with their unique mathematical foundations, has sparked new lines of investigation and propelled further advancements in the discipline.  Our proof of Theorem \ref{Th3.10} is more aligned with the techniques used by Das and Sarkar in \cite{DS}.  We note that it extends Das and Sarkar's result to vector-valued Hardy spaces $ H^2(\mathbb{D},\mathbb{C}^m)$.

\begin{lemma}\cite[Theorem 5.1]{DS}\label{Lem3.6} Let $T$ be a bounded operator on a Hilbert space $\mathcal{H}$, and let  
$\{u_i\}_{i=1}^{k}$ and $\{v_i\}_{i=1}^{k}$ be two finite sets in $\mathcal{H}$. If a subspace $\mathcal{M}\subset \mathcal{H}$ is invariant under $T -\sum_{i=1}^{k}v_i \otimes u_i$, then $\mathcal{M}$ is almost invariant under $T$ with defect at most $k$. Conversely, if $\mathcal{M}$ is almost invariant under $T$ with defect $k$, then $\mathcal{M}$ is invariant under $T -\sum_{i=1}^{k}f_i \otimes T^*f_i$ for every orthonormal basis $\{f_i\}_{i=1}^{k}$ of the defect space of $\mathcal{M}$.   
\end{lemma}

\vspace{.3 cm}

\noindent{\underline{\bf Proof of Theorem \ref{Th3.10}}}

\vspace{.2 cm}

Suppose $\mathcal{M}\subset H^2(\mathbb{D}, \mathbb{C}^m)$ is an almost $T^*_\Phi$-invariant subspace. Then using Lemma \ref{Lem3.6}, $\mathcal{M}$ is invariant under $T^*_\Phi - \sum_{i=1}^{n}F_i \otimes T_\Phi F_i$, where $\{F_i: i=1,2, \dots, n\}$ is an orthonormal basis of the defect space. Thus, by applying Theorem \ref{Th3.2}, there exist a non-negative integer $p$ and a $(T^*_\phi \otimes I_{\mathbb{C}^{p+m}})$-invariant subspace $\mathcal{K}\subset H^2(\mathbb{D},\mathbb{C}^{p+m})$ such that $$\mathcal{M}=[\mathcal{G}, I_m]\mathcal{K},$$ where $\mathcal{G}$ is the matrix with $p$ columns that form an orthonormal basis of the subspace $span\{P_\mathcal{M} T_\Phi F_i: i=1,2, \dots, n\}$,  and $I_m$ is the identity operator on $H^2(\bb D, \bb C^m)$. Furthermore, for each $F = \mathcal{G}R + H \in \mathcal{M}$,
\begin{eqnarray*}
  \|F\|^2=\|R\|^2 + \|H\|^2, \text{ where } (R,H)\in \mathcal{K}.  
\end{eqnarray*}
For the converse part, let $\mathcal{M}\subset H^2(\mathbb{D},\mathbb{C}^m)$ have the representation as given in the statement of Theorem \ref{Th3.10}. Then using Theorem \ref{Th3.6}, $\mathcal{M}$ is invariant under $T^*_\Phi-\sum_{i=1}^{p}T^*_\Phi F_i \otimes F_i$. Thus, by applying Lemma \ref{Lem3.6}, $\mathcal{M}$ is almost $T^*_\Phi$-invariant with defect atmost $p$.     

\qed

\vspace{.2 cm}

Next, we move on to the poof of Theorem \ref{Th3.13}. Recall that Theorem \ref{Th3.13} describes nearly invariant subspaces of $T_{B,B^{'}}^*$-invariant subspaces of $H^2(\bb D, \bb C^m).$ Motivated by the notion of nearly $T^*_z$-invariant subspaces of $H^2(\mathbb{D})$, in \cite{DS}, the authors introduced and studied the concept of $T^*_{z,B}$-invariant subspaces of $H^2(\mathbb{D})$. Clearly, our Theorem \ref{Th3.13} extends their study to a more general setting in vector-valued Hardy spaces $H^2(\mathbb{D},\mathbb{C}^m)$. We first recall the definition of a nealy $T^*_{\phi, \psi}$-invariant subspace of $H^2(\bb D, \bb C^m)$ and give a Lemma that records an observation essential for our proof. 

\begin{definition} Suppose $\phi$ and $\psi$ are two inner functions in $H^2(\mathbb{D})$. Then a non-zero subspace $\mathcal{M}\subset H^2(\mathbb{D},\mathbb{C}^m)$ is called nearly $T^*_{\phi, \psi}$-invariant if   
 \begin{eqnarray*}
  T^*_\Phi(\mathcal{M}\cap T_\Psi H^2(\mathbb{D},\mathbb{C}^m))\subset \mathcal{M},
\end{eqnarray*}
where $T_\Phi:=T_{\phi} \otimes I_{\mathbb{C}^m}$, $T_{\Psi}:=T_{\psi} \otimes I_{\mathbb{C}^m}$.
\end{definition}

Consider a finite Blaschke product in $H^2(\mathbb{D})$ as follows:  $$
B(z)=\prod_{i=1}^{l}\frac{z-w_i}{1-\overline{w_i}z}, \text{ where } \{w_i\}_{i=1}^{l} \subset \mathbb{D}.
$$ 
It is easy to observe that the operator $T_{\mathcal{B}}:=T_B \otimes  I_{\mathbb{C}^m}$ is an isometry on $H^2(\mathbb{D},\mathbb{C}^m)$. Using the identification of $H^2(\mathbb{D},\mathbb{C}^m)$ with $H^2(\mathbb{D})\otimes \mathbb{C}^m$, we have
\begin{eqnarray*}
H^2(\mathbb{D},\mathbb{C}^m) \ominus T_\mathcal{B} H^2(\mathbb{D},\mathbb{C}^m)
&=& H^2(\mathbb{D})\otimes \mathbb{C}^m \ominus T_BH^2(\mathbb{D})\otimes \mathbb{C}^m\\
&=&\mathcal{K}_{B} \otimes \mathbb{C}^m, \text{ where } \mathcal{K}_{B}=H^2(\mathbb{D})\ominus T_BH^2(\mathbb{D}). 
\end{eqnarray*}

This implies that the dimension of $ H^2(\mathbb{D},\mathbb{C}^m) \ominus T_\mathcal{B} H^2(\mathbb{D},\mathbb{C}^m)$ is $lm$, as the dimension of the model space $ \mathcal{K}_B$ is $l$. 

\begin{lemma}\label{Lem3.9}
 Suppose $B(z)=\prod_{i=1}^{l}\frac{z-w_i}{1-\overline{w_i}z}$ is a finite Blaschke product, 
 where $w_1, \dots, w_l\in \mathbb{D}$. Define the operator $T_\mathcal{B}:=T_B \otimes I_{\mathbb{C}^m} $. Then for a subspace $\mathcal{M}\subset H^2(\mathbb{D},\mathbb{C}^m)$, the dimension of  $\mathcal{M}\ominus(\mathcal{M}\cap T_\mathcal{B} H^2(\mathbb{D},\mathbb{C}^m))$ is 
 at most $lm$.    
\end{lemma}

\begin{proof}
Suppose $\{F_i\}_{i=1}^{r}$, $1 \le r \le lm$, is an orthonormal basis of the model space 
$\mathcal{K}_{\mathcal{B}}:=H^2(\mathbb{D},\mathbb{C}^m) \ominus T_\mathcal{B} H^2(\mathbb{D},\mathbb{C}^m)$. Let $\mathcal{F}$ denote the linear span of $\{P_\mathcal{M}F_i : i=1,2, \dots, r\}$. Then $\mathcal{F}\subset \mathcal{M}$. Also, for each $F \in \mathcal{M}\cap T_\mathcal{B} H^2(\mathbb{D},\mathbb{C}^m)$  and $1\le i\le r,$
\begin{eqnarray*}
    \langle P_\mathcal{M}F_i, F \rangle=\langle F_i,F \rangle=0, 
\end{eqnarray*}
which means $ \mathcal{F}\subset \mathcal{M}\ominus(\mathcal{M}\cap T_\mathcal{B} H^2(\mathbb{D},\mathbb{C}^m))$. To show the equality, let $F\in  \mathcal{M}\ominus(\mathcal{M}\cap T_\mathcal{B} H^2(\mathbb{D},\mathbb{C}^m))$ such that $\langle F, P_\mathcal{M}F_i \rangle=0$, for each $i$. This implies that  $\langle F,F_i \rangle=0$, which further implies that 
$F\in \mathcal{M}\cap T_\mathcal{B} H^2(\mathbb{D},\mathbb{C}^m)$. This shows that $F=0$. Hence, 
$\mathcal{M}\ominus(\mathcal{M}\cap T_\mathcal{B} H^2(\mathbb{D},\mathbb{C}^m))=\cl F$, which implies that the dimension of $\mathcal{M}\ominus(\mathcal{M}\cap T_\mathcal{B} H^2(\mathbb{D},\mathbb{C}^m))$ is at most $r$. This completes the proof.
\end{proof}

\noindent\underline{{\bf Proof of Theorem \ref{Th3.13}}}

\vspace{.2 cm}

Let $\mathcal{M}$ be a non-zero subspace of $H^2(\mathbb{D},\mathbb{C}^m)$. 
To prove the forward implication, let $\mathcal{M}$ be nearly invariant under $T^*_{B, B^{'}}$. By Lemma \ref{Lem3.9}, the subspace $\mathcal{W}:=\mathcal{M}\ominus (\mathcal{M}\cap T_\mathcal{B^{'}} H^2(\mathbb{D},\mathbb{C}^m))$ is finite dimensional. Let $\{G_1, \dots, G_p\}$ be an orthonormal basis of $\mathcal{W}.$ Now, since $\mathcal{M}$ is nearly  $T^*_{B,B^{'}}$-invariant, therefore  $T^*_{\mathcal{B}}(\mathcal{M}\cap T_\mathcal{B^{'}} H^2(\mathbb{D},\mathbb{C}^m))\subset \mathcal{M}$, where $T_{\mathcal{B}}:= T_B \otimes I_{\mathbb{C}^m}$. Equivalently, we can write
\begin{eqnarray*}
    T^*_{\mathcal{B}}\Big(I-\sum_{i=1}^{p}G_i \otimes G_i\Big)\mathcal{M}\subset \mathcal{M},
\end{eqnarray*}
that is,
\begin{eqnarray*}
     \Big(T^*_{\mathcal{B}}-\sum_{i=1}^{p} T^*_\mathcal{B} G_i \otimes G_i\Big)\mathcal{M}\subset \mathcal{M}.
\end{eqnarray*} 
This shows that $\mathcal{M}$ is invariant under $T^*_{\mathcal{B}}-\sum_{i=1}^{p} T^*_\mathcal{B} G_i \otimes G_i$. Therefore, using Theorem \ref{Th3.2}, there exists a $( T^*_B \otimes I_{\mathbb{C}^{p+m}} )$ -invariant subspace $\mathcal{N}$ of $H^2(\mathbb{D},\mathbb{C}^{p+m})$ such that  
\begin{eqnarray*}
    \mathcal{M}=[\mathcal{G}, I_m]\mathcal{N},
\end{eqnarray*}
where $\mathcal{G}$ is the matrix with columns $G_1, \dots, G_p$. Additionally, as noted in Remark \ref{uni3},  
for $F= \mathcal{G}R + H \in \mathcal{M}$, we have 
\begin{eqnarray*}
    \|F\|^2=\|R\|^2 + \|H\|^2,
\end{eqnarray*} 
and the vector-valued functions $R$ and $H$ are given by
\begin{eqnarray*}
R(z)=\sum_{n=0}^{\infty}  A_n B^n \ \text{and} \ H=\sum_{n=1}^{\infty} (P_{\mathcal{K}_{\mathcal{B}}}H_{n})B^{n-1}, 
\end{eqnarray*}
where $A_n \in \mathbb{C}^p$ and $H_n=P_{\mathcal{M}\ominus \mathcal{W}}L_n,$ for some $L_n \in \mathcal{M}$. Since $H_n \in \cl M\ominus \cl W= \mathcal{M}\cap T_\mathcal{B^{'}} H^2(\mathbb{D},\mathbb{C}^m)$ and $B$ divides $B^{'}$, therefore $H_n \in  \mathcal{M}\cap T_\mathcal{B} H^2(\mathbb{D},\mathbb{C}^m)$.
This shows that $H=0$. Thus we conclude $\mathcal{N}\subset H^2(\mathbb{D},\mathbb{C}^p)$, and it is invariant under $(T^*_B \otimes I_{\mathbb{C}^p})$, which completes the forward implication.

To prove the converse, let $F \in \mathcal{M}\cap T_{\mathcal{B^{'}}} H^2(\mathbb{D},\mathbb{C}^m)$, and let $F=\mathcal{G}R$ for some $R\in \mathcal{N}$. Then,
\begin{eqnarray*}
   0&=&\langle F, G_n \rangle\\
   &=& \langle \mathcal{G}R,\cl U E_n \rangle \\
   &=& \langle R,E_n \rangle.
\end{eqnarray*}
This shows that $R(0)=0$. Then, $R=\sum_{n=1}^\infty A_n B^n$; therefore, 
$$
F = T_{\mathcal{B}}\Big(\mathcal{G}(T^*_B \otimes I_{\mathbb{C}^p})R\Big),
$$
which yields $T^*_\mathcal{B}F=\mathcal{G}(T^*_B \otimes I_{\mathbb{C}^p})R\in \mathcal{G}\mathcal{N}=\mathcal{M}$.
This establishes that $\mathcal{M}$ is nearly $T^*_{B,B^{'}}$-invariant; hence, completes the proof.    
\qed

\vspace{.2 cm}

\begin{remark}\label{das1} 
 In \cite{DS}, Das and Sarkar describes nearly $T^*_{z,B}$-invariant subspaces of $H^2(\mathbb{D})$ \cite[Theorem 6.2]{DS}. It follows as a corollary to Theorem \ref{Th3.13} by taking $m=1$ and $B=z$.
\end{remark}

\section{Proofs of Theorems \ref{Th4.2} \& \ref{Th4.4}, and their general cases}\label{Sec4}
In this section, we shall prove Theorems \ref{Th4.2} and \ref{Th4.4} that together characterize nearly $T_{\phi, \psi}^*$-invariant subspaces of $H^2(\bb D, \bb C^m)$ with defect one. We also give the corresponding results for the general case of finite defect $n$.  

In 2020, Chalendar, Gallardo-Guti$\rm{\acute{e}}$rrez, and Partington in \cite{CGP}, introduced and studied the concept of nearly invariant subspaces with finite defect of the backward shift operator on $H^2(\mathbb{D})$. Taking a cue from this, we extend our Theorem \ref{Th3.13}, and thereby the generalization of nearly invariant subspaces due to Das and Sarkar (Remark \ref{das1}), to the finite defect setting. We first recall the definition of a nearly $T_{\phi,\psi}^*$-invariant subspace with finite defect.  

\begin{definition} Let $\phi$ and $\psi$ be two inner functions in $H^2(\mathbb{D})$. Then a non-zero subspace $\mathcal{M}$ of $H^2(\mathbb{D},\mathbb{C}^m)$ is said to be a nearly $T^*_{\phi, \psi}$-invariant with finite defect $n$ if there exists an $n$-dimensional subspace $\mathcal{F}$ (orthogonal to $\mathcal{M}$) such that 
$$
T^*_{\Phi}\Big(\mathcal{M}\cap T_{\Psi} H^2(\mathbb{D},\mathbb{C}^m)\Big) \subseteq \mathcal{M} \oplus \mathcal{F},
$$ 
where $T_{\Phi}:= T_{\phi}\otimes I_{\mathbb{C}^m}$ and $T_{\Psi}:= T_{\psi}\otimes I_{\mathbb{C}^m}$.
\end{definition}

Note that the notion of a nearly $T^*_{\phi, \psi}$-invariant subspace is a particular case of a nearly $T^*_{\phi, \psi}$-invariant subspace with finite defect. 

\vspace{.3 cm}

\noindent\underline{{\bf Proof of Theorem \ref{Th4.2}}} 

\vspace{.2 cm}

Let 
$$
T^*_{\cl B}\Big(\mathcal{M}\cap T_{\cl B^{'}} H^2(\mathbb{D},\mathbb{C}^m)\Big) \subseteq \mathcal{M} \oplus \mathcal{F},
$$
where $T_{\cl B}:=T_B\otimes I_{\bb C^m}, \ T_{\cl B^{'}}:=T_{B^{'}}\otimes I_{\bb C^m}\ $, and $\cl F$ is the defect space with dimension $1$ and orthonormal basis $\{J\}$.  

First, we consider the case when $\mathcal{M}\not\subset T_{\mathcal{B^{'}}} H^2(\mathbb{D},\mathbb{C}^m)$. Then $\mathcal{W}:=\mathcal{M}\ominus (\mathcal{M}\cap T_{\mathcal{B^{'}}} H^2(\mathbb{D},\mathbb{C}^m)) \neq \{0\}$. Let the dimension of $\mathcal{W}$ be $p$, and let $\{G_1, \dots, G_p\}$ be an orthonormal basis of $\mathcal{W}.$ We decompose $\mathcal{M}$ as 
$$
\mathcal{M}= \mathcal{W} \oplus (\mathcal{M}\ominus \mathcal{W}).
$$
Let $F \in \mathcal{M}$, then we can write it as 
\begin{eqnarray*}
F &=& \sum_{i=1}^{p}a_{0i}G_i  + P_{\mathcal{M}\ominus \mathcal{W}}F\\
&=& \mathcal{G} A_0 + P_{\mathcal{M}\ominus \mathcal{W}}F, 
\end{eqnarray*}
where $A_0=(a_{01}, \dots,a_{0p})^T \in \mathbb{C}^p$ and $\mathcal{G}=[G_1, \dots, G_p].$ Now, since $P_{\mathcal{M}\ominus \mathcal{W}}F \in \mathcal{M}\cap T_{\mathcal{B^{'}}} H^2(\mathbb{D},\mathbb{C}^m)$, therefore 
$T^*_{\mathcal{B}}( P_{\mathcal{M}\ominus \mathcal{W}}F)\in \mathcal{M}\oplus \mathcal{F}$. This 
implies that $T^*_{\mathcal{B}}( P_{\mathcal{M}\ominus \mathcal{W}}F)= L_1 + \alpha_1 J$ for some $L_1 \in \mathcal{M}$ and $\alpha_1\in \bb C$ , which further implies that $ P_{\mathcal{M}\ominus \mathcal{W}}F= B L_1 + B\alpha_1J$. Thus, we obtain
\begin{eqnarray}\label{e4.1}
    F = \mathcal{G}A_0 + B L_1 + B\alpha_1 J.
\end{eqnarray}
Taking the norms in the above equation, we get
\begin{eqnarray*}
    \|F\|^2=\|A_0\|^2 + \|L_1\|^2 + |\alpha_1|^2.
\end{eqnarray*}

Again $L_1 \in \mathcal{M}$; therefore, following the arguments similar to those used above, we get
\begin{eqnarray*}
    L_1=\mathcal{G}A_1 + B L_2 + B\alpha_2 J
\end{eqnarray*}
with
$$
\|L_1\|^2=\|A_1\|^2 + \|L_2\|^2 + |\alpha_2|^2,
$$
where $A_1\in \bb C^p, \ \alpha_2\in \bb C$, and $L_2\in \cl M$ such that $T^*_{\mathcal{B}}( P_{\mathcal{M}\ominus \mathcal{W}}L_1)= L_2 + \alpha_2 J$. Using this in Equation (\ref{e4.1}), we obtain
$$
F=\mathcal{G}(A_0 + A_1B) + B^2 L_2 +B (\alpha_1  + B \alpha_2 )J.
$$

\noindent Continuing this same process, we get that for each $n \in \mathbb{N}$,   
\begin{eqnarray}
    F =\mathcal{G}\Bigg(\sum_{i=0}^{n-1}A_i B^i\Bigg) + B^{n} L_{n} + B\Bigg(\sum_{i=1}^{n}B^{i-1}\alpha_i\Bigg)J.
\end{eqnarray}
and 
\begin{equation}\label{fin1}
    \|F\|^2=\sum_{i=0}^{n-1}\|A_i\|^2 + \|L_{n}\|^2 + \sum_{i=1}^{n}|\alpha_i|^2. 
\end{equation}
Suppose $P_1$ and $P_2$ are projections of $H^2(\mathbb{D}, \mathbb{C}^m)$ onto the orthogonal complements of $\cl F$ and $\mathcal{W}$, respectively. Then, for any $G\in \cl M, \ P_2G=P_{\cl M\ominus \cl W}G$; therefore, we have that $L_{n}=P_1T^*_{\mathcal{B}}P_2L_{n-1}$. Since $T_\mathcal{B}$ is a $C._{0}$ contraction, therefore by using Lemma \ref{Lem2.2}, $T_{\mathcal{B}}P_1P_2$ is a $C._{0}$ contraction. Thus, 
\begin{eqnarray*}
  \|L_n\|&=& \|P_1T^*_{\mathcal{B}}P_2L_{n-1}\| \\ 
  &=& \|(P_1T^*_{\mathcal{B}}P_2)^{n-1}L_1\| \\
  &=& \|P_1T_{\mathcal{B}}^*(P_2P_1T_{\mathcal{B}}^*)^{n-2}P_2L_1\| \\
  &=& \|P_1T_{\mathcal{B}}^*((T_{\mathcal{B}}P_1P_2)^*)^{n-2}P_2L_1\| \\
  &\le & \|T^*_{\mathcal{B}}\| \|((T_{\mathcal{B}}P_1P_2)^*)^{n-2}P_2L_1\| \\
  & \longrightarrow & 0 \text{ as } n \longrightarrow \infty.
\end{eqnarray*}

Since $\{A_i\}_{i=0}^\infty$ and $\{\alpha_i\}_{i=0}^\infty$ are both square summable, therefore, using the arguments similar to the ones used in the proof of Theorem \ref{Th3.2}, we conclude that   
$$
R:=\sum_{n=0}^{\infty}  A_n B^n \ \ \text{and} \ \ h:=\sum_{n=1}^{\infty}\alpha_n B^{n-1} .
$$
are well-defined  functions in $H^2(\bb D, \bb C^p)$ and $H^2(\bb D),$ respectively. Now, considering the $n^{th}$ partial sums
$$
R_n=\sum_{i=0}^{n}A_iB^i \quad \text{ and } \quad h_n=\sum_{i=1}^{n}\alpha_iB^{i-1},
$$
and using the fact that convergence in $H^2(\bb D, \bb C^p)$ and $H^2(\bb D)$ implies pointwise convergence, along with information that $L_n\to 0$, we obtain  
\begin{eqnarray}\label{e4.3}
F=\mathcal{G}R + BhJ.
\end{eqnarray}
Additionally, Equation (\ref{fin1}) yields 
\begin{eqnarray}\label{e4.4}
    \|F\|^2=\|R\|^2 +\|h\|^2.
\end{eqnarray}

We note that Equation (\ref{e4.3}) satisfies the following conditions:
\begin{enumerate}
\item[D1.] $F=\sum_{n=0}^{\infty}  A_n B^n$, $P_\mathcal{W}L_n=\mathcal{G}A_n$, where $P_\mathcal{W}$ is the projection of $H^2(\mathbb{D},\mathbb{C}^m)$ onto $\mathcal{W}$ and  $L_n=(P_1T^*_{\mathcal{B}}P_2)^nF$ for all $n \ge 0$.
\item[D2.] $h=\sum_{n=1}^{\infty}\alpha_nB^{n-1}$.
\item[D3.]  $\|F\|^2=\|R\|^2 + \|h\|^2$.
\end{enumerate}

Now, following the arguments similar to those used in the proof of Theorem \ref{Th3.2}, we show that the representation of $F$ in Equation (\ref{e4.3}), subjected to conditions D1, D2, and D3, is unique. Furthermore, the set $\mathcal{K}\subset H^2(\mathbb{D},\mathbb{C}^{p+1})$ defined by
\begin{equation}\label{rep-def1}
    \mathcal{K}=\big\{(R,h): F=\mathcal{G}R+BhJ\in \mathcal{M}, \ {\rm and} \ F, R, h \ {\rm satisfy} \ D1, D2, D3\big\}
\end{equation} 
is a closed vector subspace of $H^2(\bb D, \bb C^{p+1})$ that is invariant under $T^*_B \otimes I_{\mathbb{C}^{p+1}}$. 

\vspace{.2 cm}

The desired conclusion for the case $\mathcal{M}\subset T_{\mathcal{B^{'}}} H^2(\mathbb{D},\mathbb{C}^m)$ can easily be established by following the arguments similar to those used for the previous case, along with fact that for this case $\cl W =\{0\}$.  

Lastly, note that in this case, the subspace $\cl K$ of $H^2(\bb D)$ that appears in the representation is given by 
\begin{equation}\label{rep-def2}
\mathcal{K}=\{h: F=BhJ\in \mathcal{M}, \ h=f\circ B  \ \text{for some} \ f\in H^2(\bb D) \ {\rm and} \ ||F||=||h||\big\}.
\end{equation}
\qed

\begin{remark}
Similar to Theorem \ref{Th3.2}, the proof of Theorem \ref{Th4.2} also precisely describes the subspace $\cl K$, by Equation (\ref{rep-def1}) for case (\textit{i}) and Equation (\ref{rep-def2}) for case (\text{ii}), that appears in the representation of $\cl M$ given by Theorem \ref{Th4.2}. As a consequence, we can define a unitary from $\mathcal{K}$ onto $\mathcal{M}$ by setting $(R, h)\mapsto \mathcal{G}R+BhJ$ for case (\textit{i}) and $h\mapsto BhJ$ for case (\textit{ii}). Furthermore, by using arguments similar to those employed in Proposition \ref{uni1}, we can show that the existence of such a unitary is sufficient to guarantee that elements of $\mathcal{K}$ are completely determined by Equations (\ref{rep-def1}) and (\ref{rep-def2}) for the respective cases. 
\end{remark}

\noindent\underline{{\bf Proof of Theorem \ref{Th4.4}}}

\vspace{.2 cm}

First, we assume that $\mathcal{M}$ has the representation as given in $(i)$. Let $\{G_1, \dots, G_p\}$ be an orthonormal basis of $\mathcal{M}\ominus (\mathcal{M}\cap T_{\mathcal{B^{'}}} H^2(\mathbb{D},\mathbb{C}^m))$. We fix any $F \in \mathcal{M}\cap T_{\mathcal{B^{'}}} H^2(\mathbb{D},\mathbb{C}^m)$. Then there exists $(R,h)\in \mathcal{K}$ such that $F=\mathcal{G}R+BhJ$. 
We claim that $R(0)=0.$ For this, we shall show that $R\perp E_i$ for all $1\le i\le m$. Since $(R, h)\mapsto \cl G R +BhJ$ is a unitary, therefore  
 \begin{eqnarray*}
\langle R,E_i \rangle &=& \langle (R, h), (E_i, 0)\rangle\\
& = & \langle \mathcal{G}R +BhJ, \mathcal{G}E_i + 0\rangle \\
   &=& \langle F, G_i \rangle\\ 
   &=& 0.
\end{eqnarray*}
Thus, $R(0)=0$, which implies that $(T_{B}\otimes I_{\bb C^p})(T_{B}^*\otimes I_{\bb C^p})R=R$. Then, 
\begin{eqnarray*}
F & = & \cl GR+BhJ\\
&=& T_{\cl B}\cl G \Big((T_{B}^*\otimes I_{\bb C^p}) R\Big)+ B\Big(T_BT_B^*h+P_{\cl K_{B}}h\Big)J\\
& = & T_{\cl B}\Big(\cl G (T_{B}^*\otimes I_{\bb C^p})R + T_BT_B^*hJ+P_{\cl K_{B}}h J\Big)
\end{eqnarray*}
where $T_{\cl B}:=T_B\otimes I_{\bb C^m}$ and $\cl K_{B}=H^2(\bb D)\ominus BH^2(\bb D).$ This implies that 
\begin{eqnarray*}
T^*_{\mathcal{B}} F & = & \cl G (T_{B}^*\otimes I_{\bb C^p})R + T_BT_B^*hJ+P_{\cl K_{B}}h J\\
&=& \cl G (T_{B}^*\otimes I_{\bb C^p})R + B(T_B^*h)J+h(0) J\\ 
\end{eqnarray*}

But, $(R, h)\in \cl K$ and $\cl K$ is invariant under $T_{B}\otimes I_{\bb C^{p+1}};$ hence 
$T^*_{\cl B}(F)\in \cl M+span\{J\}$. This shows that $\mathcal{M}$ is nearly $T^*_{B,B^{'}}$-invariant with defect 1. A similar set of arguments works for the case when $\mathcal{M}$ has the representation as given in $(ii)$. This completes the proof.
\qed

The following result is an extension of Theorem \ref{Th4.2} to nearly $T^*_{B, B^{'}}$-invariant subspaces with any finite defect. The techniques required to prove it are quite similar to those employed for Theorem \ref{Th4.2}; so, we avoid its proof and record only its statement for completeness.

\begin{theorem}\label{Th4.5}
Suppose $B$ and $B^{'}$ are two finite Blaschke products such that $B$ divides $B^{'}$, and $B$ vanishes at zero. Let $\mathcal{M}\subset H^2(\mathbb{D},\mathbb{C}^m) $ be a non-zero nearly $T^*_{B,B^{'}}$-invariant subspace with defect $n$, and let $\{J_1, \dots, J_n\}$ be an orthonormal basis of the defect space.
\begin{enumerate} 
\item[(i)] If $\mathcal{M}\not\subset T_{\mathcal{B^{'}}} H^2(\mathbb{D},\mathbb{C}^m)$, then there exists a $T^*_B \otimes I_{\mathbb{C}^{p+n}}$-invariant subspace $\mathcal{K}$ of $ H^2(\mathbb{D},\mathbb{C}^{p+n})$ such that
$$
\mathcal{M}=\Bigg\{F\in H^2(\bb D, \bb C^m): F=\mathcal{G}R + B\sum_{i=1}^{n}h_i J_{i} \ \ {\rm for} 
\ (R,h_1,\dots,h_n)\in\mathcal{K}\Bigg\},
$$
where $\mathcal{G}$ is the matrix with  $p$ columns that form an orthonormal basis of $\mathcal{M}\ominus (\mathcal{M}\cap T_{\mathcal{B'}} H^2(\mathbb{D},\mathbb{C}^m))$. Further, 
for $F= \mathcal{G}R + B\sum_{i=1}^{n}h_i J_{i}\in \mathcal{M}$,
\begin{eqnarray*}
\|F\|^2=\|R\|^2 +\sum_{i=1}^{n} \|h_i\|^2.
\end{eqnarray*}

\item[(ii)] If $\mathcal{M}\subset T_{\mathcal{B^{'}}} H^2(\mathbb{D},\mathbb{C}^m)$, then there exists a $T^*_B \otimes I_{\mathbb{C}^{n}}$-invariant subspace $\mathcal{K}$ of $H^2(\mathbb{D},\mathbb{C}^n)$ such that
$$
\mathcal{M}=\Bigg\{F\in H^2(\bb D, \bb C^m): F=B\sum_{i=1}^{n}h_i J_{i} \ \ {\rm for} \ (h_1,\dots,h_n)\in \mathcal{K} \Bigg\}.
$$
Further, for $F= B\sum_{i=1}^{n}h_i J_{i}\in \mathcal{M}$,
\begin{eqnarray*}
\|F\|^2=\sum_{i=1}^{n} \|h_i\|^2.
\end{eqnarray*}
\end{enumerate}
\end{theorem}

\begin{remark}
The converse of Theorem \ref{Th4.5} is also true. Its statement is similar to Theorem \ref{Th4.4}, with pertinent changes, and it can be proved by following the arguments similar to those used in Theorem \ref{Th4.4}.    
\end{remark}

\bibliography{refs.bib}

\end{document}